\documentclass[preprint,twocolumn,5p,authoryear]{elsarticle}

  \journal{Automatica}
\usepackage[usenames, dvipsnames]{xcolor}
\usepackage{url}
\usepackage{siunitx}
\usepackage[colorlinks=true,%
		raiselinks=true,%
		linkcolor=Magenta,%
		citecolor=MidnightBlue,%
		urlcolor=ForestGreen,%
		pdfauthor={Who's that?},%
		pdftitle={What's this?},%
		pdfkeywords={},%
		pdfsubject={chapter},%
		plainpages=false]{hyperref}

\usepackage{mathtools,amsmath,amssymb,amsfonts,amsthm}
\usepackage[bitstream-charter]{mathdesign}
\usepackage{amsfonts}
\usepackage{mathalfa}
\usepackage{esvect}
\usepackage{caption}
\usepackage{subcaption}
\usepackage{comment}
\usepackage{tikz}
\usepackage{pgfplots}
\pgfplotsset{compat=1.12}

\theoremstyle{plain}
\newtheorem{thm}{Theorem}[section]
\newtheorem{prop}[thm]{Proposition}

\theoremstyle{remark}
\newtheorem{rem}[thm]{Remark}
\theoremstyle{definition}

\newtheorem{theorem}{Theorem}
\newtheorem{lemma}{Lemma}
\newtheorem{assume}{Assumption}
\newtheorem{proposition}{Proposition}

\newtheorem{remark}{Remark}

\newcommand{\I}{\mathrm{I}}

\newcommand{\weak}{\rightharpoonup}

\newcommand{\hmu}{\hat{\mu}}
\newcommand{\hmut}{\hat{\mu}^{\tau}}

\newcommand{\hvt}{\hat{v}^{\tau}}

\newcommand{\p}{{\rm proj}}
\newcommand{\mul}{\mu^{\lambda}_t}
\newcommand{\xl}{x^{\lambda}}
\newcommand{\xn}{x^{\nu}}
\newcommand{\yl}{Y_{\lambda}(x^{\lambda}(t))}
\newcommand{\yn}{Y_{\nu}(x^{\nu}(t))}

\newcommand{\cB}{\mathcal{B}}
\newcommand{\cX}{\mathcal{X}}
\newcommand{\cM}{\mathcal{M}}
\newcommand{\N}{\mathbb{N}}
\newcommand{\cR}{\mathcal{R}}
\newcommand{\cL}{\mathcal{L}}
\newcommand{\cW}{\mathcal{W}}
\newcommand{\cC}{\mathcal{C}}
\newcommand{\dd}{\mathrm{d}}
\newcommand{\cP}{\mathcal{P}}
\newcommand{\R}{\mathbb{R}}
\newcommand{\cN}{\mathcal{N}}

\renewcommand{\P}{\mathcal{P}}

\newcommand{\bx}{\mathbf{x}}
\newcommand{\bv}{\mathbf{v}}
\newcommand{\rhoo}{{\mu}^{\tau}_{k+1}}
\newcommand{\rhot}{{\mu}^{\tau}_{t}}
\newcommand{\rhoe}{{\mu}^{\tau}_{k}}

\renewcommand{\L}{\mathcal{L}}
\newcommand{\ul}[1]{\underline #1}
\newcommand{\ol}[1]{\overline #1}
\newcommand{\supp}{\mathrm{supp}}

\newcommand{\wt}{\widetilde}

\DeclareMathOperator{\id}{id}

\begin{document}

\begin{frontmatter}

\title{Evolution of Measures in Nonsmooth Dynamical Systems:\\ Formalisms and Computation\tnoteref{anr}}
\author[Laas]{Saroj Prasad Chhatoi}\ead{spchhatoi@laas.fr}
\author[Laas]{Aneel Tanwani}\ead{tanwani@laas.fr}
\author[Laas,Czech]{Didier Henrion}\ead{henrion@laas.fr}
%\tnotetext[anr]{{\color{red} This work was supported by the ANR project.? which one}}
\address[Laas]{LAAS -- CNRS, University of Toulouse, France.}
\address[Czech]{Faculty of Electrical Engineering, Czech Technical University in Prague, Czechia}

\date{} 
\begin{abstract}
	This article develops mathematical formalisms and provides numerical methods for studying the evolution of measures in nonsmooth dynamical systems using the continuity equation. The nonsmooth dynamical system is described by an evolution variational inequality and we derive the continuity equation associated with this system class using three different formalisms. The first formalism consists of using the {superposition principle} to describe the continuity equation for a measure that disintegrates into a probability measure supported on the set of vector fields and another measure representing the distribution of system trajectories at each time instant. The second formalism is based on the regularization of the nonsmooth vector field and describing the measure as the limit of a sequence of measures associated with the regularization parameter. In doing so, we obtain quantitative bounds on the Wasserstein metric between measure solutions of the regularized vector field and the limiting measure associated with the nonsmooth vector field. The third formalism uses a time-stepping algorithm to model a  time-discretized evolution of the measures and show that the absolutely continuous trajectories associated with the continuity equation are recovered in the limit as the sampling time goes to zero. We also validate each formalism with numerical examples. For the first formalism, we use polynomial optimization techniques and the moment-SOS hierarchy to obtain approximate moments of the measures. For the second formalism, we illustrate the bounds on the Wasserstein metric for an academic example for which the closed-form expression of the Wasserstein metric can be calculated. For the third formalism, we illustrate the  time-stepping based algorithm for measure evolution on an example that shows the effect of the concentration of measures.
\end{abstract}

\begin{keyword}
	%Superposition principle\sep 
	Nonsmooth dynamical systems, Optimal transport, Polynomial optimization.
\end{keyword}
\end{frontmatter}

\section{Introduction}
The study of evolution of measures in finite dimensional systems has found relevance in the design of optimal control problems, understanding the system behavior under uncertainties, and several other applications. The primary step in this direction is to understand how the probabilistic initial conditions evolve in time under the action of a vector field. Such questions have been fairly well studied for single-valued dynamical systems with sufficient regularity (such as Lipschitz continuity) of the vector field. However, when we relax the regularity assumptions on the vector field, the question of evolution of measures brings forth some interesting questions which are of relevance for the applications as well. We are thus motivated to study the evolution of probability measures for a class of dynamical systems described by differential inclusions and in particular where the differential inclusion models the trajectories constrained to a pre-specified set. We present different mathematical formalisms to study measure evolution for such dynamical systems and provide corresponding numerical algorithms for simulations.

%\subsection{Literature review}
For an autonomous dynamical system described by an ordinary differential equation (ODE) with Lipschitz continuous vector field, the time evolution of the measure describing the initial condition is governed by a linear partial differential equation (PDE), commonly called the continuity equation or the Liouville equation \cite[Section 5.4]{Vil03}. The solution to this PDE, that is the probability measure describing the distribution at a given time, is the push-forward or image of the initial probability measure through the flow map at that time. Lipschitz continuity of the vector field ensures that the flow map of the ODE is invertible, which in turn ensures that the push-forward measure is the unique solution to the continuity equation. The Cauchy problem for continuity equation with Sobolev fields was studied  by \citep{DiPeLion89}. Continuity equations corresponding to one-sided Lipschitz vector fields have been studied in \citep{BoucJame98, BoucJameManc05}. In \citep{Amb08} and \citep{AmbrGigl05}, the authors consider fields of bounded variation and discuss the potential nonuniqueness of solutions to the continuity equation by introducing the notion of {superposition principle}. For the differential inclusions with convex set-valued mappings, the reference \citep{spDI} provides a generalized superposition principle. For our purposes, the solutions based on the superposition principle are useful for numerical purposes. We propose a vector field selection from a time-varying differential inclusion from which we derive a continuity equation suitable for numerical algorithms. We use the converse statement of the superposition principle to characterize all possible solutions to the proposed continuity equation.

In this article, we are particularly concerned with a class of dynamical systems where the nonsmoothness arises due to the modeling of constraints on state trajectories. Such systems are described by the inclusion
\begin{equation}\label{eq:sysConstrained}
\dot x (t) \in f(x(t)) - \cN_{S(t)}(x(t))
\end{equation}
where $\cN_S (x) \in \R^n$ denotes the outward normal cone to the set $S$ at the point $x \in \R^n$. Since the normal cone takes a zero value in the interior of $S$, it is clear that the right-hand side of \eqref{eq:sysConstrained} is potentially discontinuous at the boundary of the set $S$. One can also think of \eqref{eq:sysConstrained} as an evolution variational inequality, described as
\[
\left\langle \dot x(t) - f(x(t)), y - x(t) \right\rangle \ge 0,
\]
for all $y \in S$, $x(t) \in S$, $t \in [0,T]$, where the brackets denote the inner product between vectors.
%Such dynamical systems have been a matter of extensive study in past decades due to their relevance in engineering and physical systems. 
The models studied in this paper are particularly relevant for systems in mechanics and electronics with nonsmooth effects \citep{BroTan20}. The survey article \citep{BroTan20}, and a research monograph \citep{Adly2018}, provide an overview of different research oriented directions in the literature pertaining to system \eqref{eq:sysConstrained} and its connections to different classes of nonsmooth mathematical models. Analysis of such systems requires tools from variational analysis, nonsmooth analysis, set-valued analysis \citep{aubin1990, Mord06, RockWets98}. For a fixed initial condition, $x(0) \in S$, the question of existence and uniqueness of solution to system \eqref{eq:sysConstrained} has already been well-established in the literature, and the origins of such works can be found in \citep{moreau1977}, see \citep{EdmoThib06} for a recent exposition. 

However, if we consider the initial conditions described by a probability measure, then the evolution of this measure under the dynamics of \eqref{eq:sysConstrained} has received much less attention in the literature. One can study such problems by considering stochastic versions of \eqref{eq:sysConstrained} by adding a diffusion term on the right-hand side. 
Such systems first came up in the study of variational inequalities arising in stochastic control \citep{BensLion78}, and in the literature, we can find results on existence and uniqueness of solutions in appropriate function space. In \citep{Cepa95}, this is done by considering Yosida approximations of the maximal monotone operator, whereas \citep{Bern03} provides a proof based on time-discretization of system \eqref{eq:sysConstrained}. These approaches have been generalized for prox-regular set $S$ in \citep{bernicot2011}, and the case where the drift term contains Young measures \citep{castaing2014, castaing2016}.  In \citep{marino2016}, the authors provide a constructive approximation of measures associated with system \eqref{eq:sysConstrained} with $f \equiv 0$, which are based on a generalization of time-stepping algorithm and involves projecting the density function onto the constraint set with respect to the Wasserstein metric. %One could also, in principle, formulate a partial differential equation with set-valued elements and study the solutions of such equations under appropriate hypothesis, which is the case in \citep{BonnFran21} but it is not clear how to derive the corresponding set-valued partial differential equation for system~\eqref{eq:sysConstrained} and whether the resulting inclusion would satisfy the necessary hypothesis for well-posedness. 

The main contribution of this article is to provide different formalisms for describing the evolution of measures for the class of systems considered in \eqref{eq:sysConstrained}. In particular, our contribution lies in studying three different techniques for describing the propagation of probabilistic initial conditions for system~\eqref{eq:sysConstrained} and we provide numerical methods for each of these techniques. 

The first approach is based on using the previously mentioned superposition principle. Here, we consider a continuity equation where the velocity vector field is obtained by a selection of the set-valued mapping in system dynamics, which results in a (possibly non-unique) solution to the measure evolution. We develop a converse result which actually shows that all possible solutions can be associated with a selection of the vector field. The tools used in the process are similar to the ones appearing in \citep{spDI}, but we develop a specific representation of the continuity equation in terms of a measure which can be computed numerically using the moment-SOS (polynomial sums of squares) hierarchy \citep{henrion2020moment} and semi-definite programming based techniques.

The second approach builds on our recent work in \citep{SOUAIBY2023110836} where we  approximate the dynamics of system~\eqref{eq:sysConstrained} by ODEs with Lipschitz continuous right-hand side. The solution of the continuity equation associated with each ODE provides a sequence of measures which allows us to approximate the solution of the measure evolution problem. We show that the limiting measure can be represented by the pushforward of the unique flow map of system \eqref{eq:sysConstrained} and we develop quantitive bounds on the Wasserstein distance between the limiting measure and its approximations obtained from the regularization method.
The results are of independent interest and also are essential in providing estimates on the approximation equality and convergence rates of numerical schemes based on this formalism.

Another approach we adopt for studying the evolution of measures subject to constrained dynamics is based on computing an approximation of the transport maps for system~\eqref{eq:sysConstrained} via time discretization. Time discretization based techniques are well known for constructing solutions to evolution PDEs using the gradient flow structure of the Wasserstein space \citep{filipoOT}. Time discretization schemes have been recently used in \citep{aneel}, which exploits the gradient flow structure for the system class \eqref{eq:nsd} in the Euclidean space to construct solutions for constrained optimization problems.
For sweeping processes without the perturbation term, this approach was adopted in \citep{marino2016} and it generalizes the classical time-stepping algorithm proposed in \citep{moreau1977} to the setting of measures. We use these techniques to construct the solutions of the continuity equation associated with system \eqref{eq:sysConstrained}.
In particular, one computes the distribution at discrete time instants by interpolating the distribution through the perturbation term, and then projecting it onto the constraint set with respect to the Wasserstein metric. This scheme is built on the dynamic viewpoint of optimal transport problems where the absolutely continuous curves in Wasserstein space satisfy the continuity equation. Under certain conditions, we show that the sequence of time-discretized measures converges to a solution described by the push-forward of the initial distribution under the transport map.

Finally, we address the computational aspects for each of the three formalisms with the help of academic examples. The proposed continuity equation could be seen as an infinite dimensional linear program problem in the space of measures. We use the moment-SOS hierarchy to approximate the moments of the measures and we provide an illustration of this method. For the second formalism based on functional regularization, we explicitly, for a one dimensional case, compute the Wasserstein distance between the measure solutions to the nonsmooth system and to approximation obtained by regularization method. For the last formalism based on time-stepping algorithm for measures, we consider an example of a two-dimensional system based on time-space discretization and then evolving the measures using the given algorithm.

\section{Preliminaries and Overview}

\subsection{Measure Evolution}

Consider the dynamical system described by an ordinary differential equation (ODE):
\begin{equation}\label{eq:sysODE}
\dot{x} = f(t,x).
\end{equation}
If the vector field $f:\R_{\ge 0} \times \R^n \to \R^n$ is such that $f(\cdot, x)$ is Lebesgue measurable for each $x \in \R^n$ and $f(t, \cdot)$ is Lipschitz continuous for each $t \in \R\ge 0$, then there exists a unique absolutely continuous function $x:\R_{\ge 0} \to \R^n$ that solves \eqref{eq:sysODE}. Consequently, we consider the {\it flow map} $X_t:\R^n \to \R^n$ parameterized by $t\in \R_{\ge 0}$ having the property that $x(t) = X_t(x_0)$ for each $x_0 \in \R^n$.
It is also of interest to study the evolution of probability measures for system~\eqref{eq:sysODE} when the initial condition is described by a probability distribution on $\R^n$, that is, $x(0)\sim \mu_0$, where $\mu_0 \in \cP(\R^n)$, the set of probability measures on $\R^n$. In words, the law of the random variable $x(0)$ is the probability measure $\mu_0$. The resulting measure $\mu_t \in \cP(\R^n)$, for $t \in \R_{\ge 0}$, is defined by the {\it continuity equation}, also called the Liouville equation, a linear partial differential equation (PDE) which models the transport of a distribution along the flow of trajectories of the underlying system and preserves the mass of the distribution. For the cases where the vector field $f(t,x)$ is Lipschitz in $x$ for each $t$, the continuity equation reads
\begin{equation}\label{eq:transportODE}
\partial_t \mu_t + \nabla \cdot (f(t,\cdot) \mu_t) = 0
\end{equation}
where $\nabla\cdot $ is the divergence operator. The equation is to be understood in the weak sense, i.e. $\int_{[0,T]\times \R^n} [\partial_t\varphi(t,x) + \nabla_x \varphi(t,x) \cdot f(t,x)] \dd \mu(t,x) = \int_{\R^n} \varphi(T,x)\dd \mu_T(x) - \int_{\R^n} \varphi(0,x)\dd \mu_0(x)$ for every compactly supported $\varphi\in \cC^1([0,T]\times \R^{n})$, where $\nabla_x$ is the gradient operator. Furthermore, a measure $\mu_t$ solving \eqref{eq:transportODE} can be represented as the push-forward of $\mu_0$ under the mapping $X_t$, denoted  $\mu_t = {X_t}_{\#}\mu_0$. Here, and throughout this article, for a function $g:\R^n \to \R^m$ and a measure $\mu_0$ supported on a set in $\R^n$ the push-forward of $\mu_0$ under the mapping $g$ is denoted by $g_{\#} \mu_0$ and it is defined as ${g}_{\#} \mu_0(A): = \mu_0( \{ x \in R^n : g(x) \in A \}  )$ 
for every measurable set $A \subset \R^m$.

%We consider that $X_0 \sim \P(\R^m)$ is a probability measure which specifies the initial state distribution. At each time $t$, the state satisfying \eqref{eq:nsd} will be a probability measure. We would like to study here how the distribution evolves with time which can be formulated as a problem of transport of measure such that  the measures satisfies continuity equation associated with \eqref{eq:nsd}. In this work we present various ways to address this problem of evolution of measure through \eqref{eq:nsd}. 

In this work, we particularly consider the class of following differential inclusions:
\begin{gather}\label{eq:nsd}
\dot{x}(t) \in f(t,x) - \cN_{S(t)}(x), \quad x(0) \sim \mu_0
\end{gather}
where $f:[0,T]\times \R^n \to \R^n$ is a vector field, $S:[0,T]\rightrightarrows \R^n$ is a set-valued mapping, and $\cN_{S(t)}(x)$ denotes the outward normal cone to the convex set $S(t)$ at $x \in S(t)$. We impose the following assumptions on system class \eqref{eq:nsd} so that the system is well-posed.

 %A normal cone to the set $S$ at point $x \in S$  is defined by
%\[ 
%\cN_{S}(x)\ \coloneqq \{ y\in\R^n , \langle y , x' -x \rangle \le 0, ~~\forall x' \in S \} .
%\]
%
%Given the existence of selection principle $\zeta \in \cN_{S}(x)$, an absolutely continuous solution $x:[0,T]\to \R^n$ to \eqref{eq:nsd} satisfies the following equation:
%\begin{eqnarray}
%\dot{x}(t) = f(t,x)+\zeta(t), \textnormal{where} ~ \zeta(t) \in \cN_{S(t)}(x(t))
%\end{eqnarray}
%almost everywhere in $t\in [0,T]$.

\begin{assume} \label{assume:1}
	There exists $L_f > 0$ such that
	\begin{gather*}
	|f(t,x_1)| \le L_f(1+|x_1|), \\
	|f(t,x_1) - f(t,x_2)| \le L_f|x_1 - x_2|
	\end{gather*}
	for all  $x_1,x_2 \in \R^n$.
\end{assume}
\begin{assume}\label{assume:2}
	The mapping $S:[0,T] \rightrightarrows \R^n$ is closed and convex-valued for each $t \in [0,T]$, and $S(\cdot)$ varies in a Lipschitz continuous manner with time, i.e., there exists a constant $L_s$ such that\begin{equation*}
	d_H(S(t),S(s)) \le L_s |t-s|
	\end{equation*}
	where $d_H(A,B) := \max\Big\{\sup_{x\in B} {\rm dist}(x,A), ~\sup_{x\in A} {\rm dist}(x,B)\Big\}$ is the Hausdorff distance between the sets A and B.
\end{assume}

Under these two assumptions, several references in the literature prove the existence and uniqueness of solutions to \eqref{eq:nsd} with $x(0) \in S(0) \subset \R^n$, see for example \citep{BroTan20} for an overview. In this article, we are interested in studying the evolution of measures for system class \eqref{eq:nsd}. The PDE considered in \eqref{eq:transportODE} cannot be readily obtained in that case and we study three different principles to describe the evolution of measures for our system \eqref{eq:nsd}. In the remainder of this section, we provide an overview of these techniques from the existing literature. In the later sections, we develop each of these techniques for system class \eqref{eq:nsd}.

\subsection{Superposition Principle}\label{subsec:sp}

In the first instance, we look at \eqref{eq:nsd} as a differential inclusion with a set-valued right-hand side in the dynamics. In this regard, we see that the evolution of measures is described using the {\it superposition principle} for the differential inclusions of the form
\begin{equation}\label{eq:genDI}
\dot x(t) \in F(t,x(t))
\end{equation}
where $F:\R_{\ge 0} \times \R^n \rightrightarrows \R^n$ is a set-valued mapping. Let us explain briefly and informally what is the superposition principle. A {\it selection} of $F$ is a mapping $(t,x) \mapsto \ol f(t,x) \in F(t, x)$. Associated with a selection is an absolutely continuous solution $\gamma \in AC([0,T];\R^n)$ with $\gamma(0) = x(0)$ such that $\dot \gamma (t) = \ol f(t,\gamma(t))$ for Lebesgue a.e. $t \in \R_{\ge 0}$. Let us consider the set of all admissible curves
\[
\Gamma_T := \{ \gamma \in AC([0,T];\R^n) : \dot{\gamma}= \ol f(t,\gamma), \ol f \text{ a selection of } F\}.
\]
The evaluation map is defined as a Borel measurable map 
$e_t: \R^n \times \Gamma_T \to \R^n$ such that
\begin{gather}\label{eq:eval_op} 
e_t(x,\gamma) := \gamma(t) ~ \forall t \in [0,T] {\rm ~and~}  \gamma(0) = x , \gamma \in \Gamma_T.
\end{gather}
Let $\eta$ be a probability measure such that $\eta \in \mathcal{P}(\R^n\times\Gamma_T)$. Under some mild integrability condition \cite[Theorem 8.2.1]{AmbrGigl05}, the measure solutions $\mu_t$ to a continuity equation associated with \eqref{eq:genDI} (under some selection of vector field from $F(t,x)$) can be represented as 
\begin{gather}\label{eq:repForm}
\mu_t = {e_t}_{\#}\eta
\end{gather}
which for any continuous function $\phi :\R^n \to \R$ satisfies $\int  \phi(x) \dd\mu_t(x)= \int \phi(e_t(x,\gamma))d\eta(x,\gamma).$
The solutions $\mu_t$ can be understood as a superposition over solution trajectories $\gamma \in \Gamma_T$, where the superposition is captured by the measure $\eta$.
The solutions to differential inclusion~\eqref{eq:genDI} are possibly nonunique and hence $\mu_t$ in \eqref{eq:repForm} is also not necessarily unique for a given initial measure.% In the cases where there is a unique solution trajectory, the measure $\mu_t$ is the unique solution to continuity equation having the representation $\mu_t = {X_t}_{\#}\mu_0$. 

In this work, we are interested in using the superposition principle for deriving a continuity equation associated with~\eqref{eq:nsd}. For differential inclusions, such problems have been studied in \citep{spDI}, but in comparison, we consider a specific class of non-compact time varying differential inclusions, and we derive a different form of  continuity equation which is more suitable for numerical purposes discussed later in this paper.

\subsection{Functional Regularization}

The basic idea of the regularization is to consider a sequence of ODEs with a parameter $\lambda$:
\[
\dot x^\lambda(t) = g_t^\lambda(x^\lambda(t))
\] 
so that the solutions $x^\lambda(t)$ approach the solution $x(t)$ that solves \eqref{eq:nsd}, under the constraint $x^\lambda(0) = x(0)$. Here, for each $\lambda > 0$ and for each $t \in [0,T]$, $g_t^\lambda : \R^n \to \R^n$ is a single-valued Lipschitz continuous function, whose construction is provided in Section~\ref{sec:fun_reg}. One can derive the classical continuity equation \eqref{eq:transportODE} to these ODEs and obtain a parameterized sequence of measures $\mu_t^\lambda$ as follows:
\[
\partial_t \mu_t^\lambda + \nabla \cdot (g_t^\lambda(\cdot) \mu_t) = 0.
\]
An obvious candidate for describing the measure solving \eqref{eq:nsd} is to take the limit of $\{\mu_t^\lambda\}$ as $\lambda \to 0$. In Section~\ref{sec:fun_reg}, we study the limit of this sequence using the Wasserstein metric to quantify the distance between $\mu_t^\lambda$ and the limiting measure.

To provide some background on this performance metric used to study convergence of measures, we recall that the {\it Wasserstein metric}, also called the Kantorovich-Rubenstein metric, is frequently employed to describe  the distance between two probability measures. The more common choice, the  2-Wasserstein distance between two probability measures $\mu,\nu \in \cP(\Omega)$ for some $\Omega \subset \R^n$, is defined as
\begin{gather}
W_2(\mu,\nu) := \frac{1}{2}\min_{\theta \in \Theta(\mu,\nu)} \left(\int_{\Omega\times \Omega} |x-y|^2 \dd\theta(x,y)\right)^{1/2}
\end{gather}
where $\Theta(\mu,\nu)$ is the set of joint probability measures on $\Omega \times \Omega$ with given marginals $\mu$ and $\nu$, i.e. such that $\theta \in \Theta(\mu,\nu)$ satisfies 
%$\int_{A} \int_{\Omega}\theta(x,y) \dd x \dd y = \mu (A)$ and $\int_{\Omega}\int_{A}\theta(x,y) \dd x \dd y = \nu(A)$, for every measurable $A \subset \Omega$.
$\int_{A \times \Omega}\dd \theta(x,y) = \mu(A)$ and $\int_{\Omega \times A}\dd \theta(x,y) = \nu(A)$, for every measurable $A \subset \Omega$.

Similarly, 1-Wasserstein distance or $W_1(\mu,\nu)$ is defined as,
\begin{gather}\label{eq:wass1d}
W_1(\mu,\nu) = \min \left\{ \int_{\Omega\times \Omega} |x-y| \dd\theta(x,y) : \theta \in \Theta(\mu,\nu)\right\}.
\end{gather}
\subsection{Time Discretisation and Optimal Transport} \label{subsec:ot}
Time discretization based techniques are well known for constructing solutions to evolution PDEs by using the gradient flow structure of the Wasserstein space.  
It is based on partitioning a time interval into finitely many nodes (discrete times) and describing the measure at those times  as a function of the initial distribution through appropriate mappings using the system data.
The interpolation between the two measures (described at two consecutive times) is based on the principles of optimal transport and provides an approximation to the measure evolution problem for system~\eqref{eq:nsd}. 

To provide some background on these interpolation schemes, we recall that the original mass transportation for measures was proposed in \citep{monge} as the problem of finding a transport map $G:\R^n \to \R^n$ such that given two probability measures $\mu\in \P(\R^n)$ and $\nu \in \P(\R^n)$ and a cost function $c: \R^n \times \R^n \to [0,\infty)$, it solves
\begin{gather}
\inf_G \left\{ \int c(x,G(x)) \dd\mu(x) :  G_{\#} \mu = \nu \right\}.
\end{gather}
The problem is highly nonlinear with nonconvex constraints and the existence of a minimizer is difficult to prove. The problem was later reformulated by Kantorovich \citep{ksOT} into a convex program that corresponds to the computation of $W_2^2(\mu,\nu)$ by taking $c(x,y) = \frac{1}{2}|x-y|^2$.
Indeed, $W_2(\cdot,\cdot)$ provides a metric structure to the space of measures $\P(\R^n)$ and the resulting subspace of $\P(\R^n)$ is known as the Wasserstein space $\cW_2(\R^n)$. One interesting property which will be of interest is that any absolutely continuous curve in Wasserstein space $\cW_2$ is a solution to a continuity equation \citep{filipoOT}. In \citep{McCann1997ACP} the authors proved that if there exists a pair of measures $\mu_0,\mu_1 \in \P(\R^n)$ with $\mu_0$ absolutely continuous with respect to the Lebesgue measure, then there exists a constant speed geodesic between these measures and such constant speed geodesics satisfy a continuity equation. It is possible to construct an approximation of an absolutely continuous curve by defining measures at discrete time instants and using an interpolation via constant speed geodesics between successive time instants. In \citep{marino2016}, the authors use time-discretization to approximate the measure solution of continuity equation associated with \eqref{eq:nsd} without the drift term $f(\cdot)$. The method is based on recursively defining measures at different time instants using an optimal transport map which transports the measures from one time instant to the next. Considering suitable interpolation schemes, one constructs the trajectory and shows that it converges to the solution of the continuity equation. We will use a time-stepping scheme for \eqref{eq:nsd} to construct measures at different time instants starting from an initial distribution. Using appropriate interpolation, we will prove that the interpolated curves converge to the absolutely continuous curves which will be the measure valued solutions to the continuity equation associated to \eqref{eq:nsd}.

\begin{comment}
{\color{red}For any feasible transport map $G$ we can associate a transport plan $\theta_G = (\id \times G)_\# \mu$. Under some mild conditions, Knott-Smith criterion \citep{Smith1987NoteOT} states the existence and concentration of optimal transport plans for the Kantorovich problem on the subgradients of a convex function. Due to \citep{brenier}, one can obtain, under the assumption that $\mu$ is absolutely continuous w.r.t.~Lebesgue measure, a unique optimal transport plan and a unique transport map $G$ which is a gradient of a convex function. In \cite{McCann1997ACP} the authors proved that if there exists a pair of measures $\mu_0,\mu_1 \in \P(\Omega)$ with $\mu << \L^d$ (Lebesgue measurable), then there exists a constant speed geodesic between these measures $\mu_t = ((1-t)id + tT)_{\#} \mu_0 ~{\rm for}~ t\in [0,1]$ in $W_2(\Omega)$ and this is known as the McCann interpolation. These constant speed geodesics satisfy a continuity equation. Given an absolutely continuous curve, one can construct a discrete time-approximation of this curve which consists of constant speed geodesics between successive time instants. Thus, upon showing that the discrete time approximation converges to the absolutely continuous curve ($\mu_t$) one concludes that for any absolutely continuous curves there exists a velocity such that the pair in the space of probability measures with $W_2$ metric satisfy continuity equation.}
\end{comment}

\section{Superposition Principle}\label{sec:ce_di}
In this section, we consider a general system class described by a differential inclusion. Starting from a vector field selection of this differential inclusion, we propose a continuity equation driven by this selection and we characterize all possible solutions to this equation. 
\subsection{Selection in Differential Inclusions using Measures}
Consider a dynamical system governed by the differential inclusion:
\begin{align}\label{eq:di}
\dot{x} \in F(t,x)
\end{align}
where $F: [0,T] \times \R^n \rightrightarrows \R^n$ is a set-valued mapping. For an initial condition $x_0 \in \R^n$, we denote the solution to \eqref{eq:di} at time $t\in [0,T]$ by $X_t(x_0)$, where $X_t$ represents the flow map for system~\eqref{eq:di}. For solutions to be well-defined, $F$ satisfies the following:
\begin{assume} \label{assume:3}
	The set $F(t, x)$ is convex for every $t\in [0,T]$ and every $x \in \R^n$.
\end{assume}
\begin{assume} \label{assume:4}
	If there exists a solution to \eqref{eq:di} corresponding to a selection $\overline f$ of $F$, then it holds that,
	\begin{equation}
	|\overline f (t,x)|\le \beta(t)(1+|x|) 
	\end{equation}
	where, $\beta (\cdot) \in \cL^1([0,T];\R_{+})$.
\end{assume}
In what follows, we consider a selection $\overline f_\omega(t,x)$ of $F(t,x)$ defined using a probability measure $\omega(\cdot \vert t,x)\in \cP(F(t,x))$ as follows:\footnote{ If Assumptions \ref{assume:3} and \ref{assume:4} hold, there exists a measurable selection $\overline{f}_{\omega}(t,x) \in F(t,x)~ \forall (t,x)$ \citep{aubin_di}. Measurability of the proposed vector field can be checked by first replacing the integrand with the indicator function $I_S$, then $\overline{f}_{\omega}(t,x) = \omega(S|t,x)$ which is a measurable function for every fixed $S\subset F(t,x)$. One can use standard measure theoretic arguments to approximate the integral using simple functions.%The reader can refer standard measure theory textbooks for these arguments.
}
\begin{align}\label{eq:mean_field}
\overline{f}_{\omega}(t , x) := \int_{F(t,x)} v \dd\omega (v|t,x), \quad \omega(\cdot|t,x) \in \P(F(t,x)).
\end{align}
Due to  the convexity of $F(t,x)$, it follows that $\overline{f}_{\omega}(t,x)\in F(t,x)$. We do not make any further assumptions on the regularity of $\overline{f}_{\omega}(t,x)$ and thus an ODE system $\dot{x}= \overline{f}_{\omega}(t,x)$ may admit multiple solution trajectories from a given initial condition. We also let $\Gamma_T^\omega$ denote the set of trajectories associated with the selection $\overline f_\omega $, that is,
\begin{equation}\label{eq:defGammaSub}
\Gamma_T^\omega := \{ \gamma \in AC([0,T];\R^n) : \dot{\gamma}= \ol f_\omega(t,\gamma)\}.
\end{equation}

%\textit{Note:} Later when we propose a convex optimization framework for simulation of measures evolution, $\omega$ will be an unknown as well.
\begin{comment}
Later in remark \ref{remark:sp_di}, we give an interpretation of the proposed vector field selection as the weighted average (at each time $t \in [0,T]$) of the vector fields associated with solution trajectories of \eqref{eq:di}.
\end{comment}
\begin{remark}
	In the case where the set $F(t,x)$ is finitely generated (that is, for each $(t,x)$, it is represented by a linear combination of finitely many vector fields $f_i(x)$ for $i \in \{1,...,n\}$), the vector field \eqref{eq:mean_field} reduces to a {\it Fillipov differential inclusion}. For example, given a piecewise smooth system $\dot{x}(t) = f_i(x(t))$ for $x \in \cR_i$, where $\cR_i$ are disjoint regions covering $\R^n$, the Fillipov differential inclusion for such system would result in $\dot{x}(t) = {\rm conv}(f_i(x(t)))$, where ${\rm conv}$ denotes the convex combination of the vector fields. Any absolutely continuous solution $x(t)$ would satisfy $\dot{x}(t) = \sum_{i\in I}w_i(x(t)) f_i(x(t))$ where $I(x)$ denotes the active set at $x \in \R^n$, and the weights are such that $\sum w_i = 1$ and $w_i \ge 0$. In \citep{stewart_LP}, the authors propose a linear program to compute the weights and thus solve the differential equation using an \textit{active-set} method. In this case of piecewise smooth vector fields, the measure $\omega$ will be discretely supported on the set $\{f_i(x)\}_{i \in I(x)}$,  and the vector field in \eqref{eq:mean_field} yields
	\[ 
	\overline{f}_{\omega}(x) = \sum_{v\in\{f_i(x)\}, i \in I(x)} v w_i(x) = \sum_{i\in I(x)} f_i(x) w_i(x).
	\]
	Thus, $w = (w_1, \ldots ,w_m)$ with $\sum_i w_i = 1$ can be seen as the discrete version of the measure $\omega$.
\end{remark}
\subsection{Describing Vector Field from Solutions}
The vector field selection in \eqref{eq:mean_field} is used to define the continuity equation for the measure evolution problem. Before doing so, we make some connections with related literature
to provide an interpretation of $\overline f_\omega(t,x)$ as the weighted average of the vector fields  associated with solution trajectories of \eqref{eq:di}.
In \citep{spDI}, the authors show that the image measure $\mu_t $ obtained by applying the evaluation map to $\eta \in \P(\R^n \times \Gamma_T)$, as described in \eqref{eq:repForm}, are solutions to a continuity equation driven by a \textit{mean-vector field}. By definition of evaluation operator \eqref{eq:eval_op}, we define the set $e_t^{-1}(x)$ of trajectories passing through $x$ at time $t$, i.e.,
\begin{gather}
e_t^{-1}(x) := \{ (y,\gamma) ~{\rm s.t.}~ \gamma \in \Gamma_T, \gamma(0) = y, \gamma(t) = x \}.
\end{gather}
A disintegration $\eta_{t,x} (y,\gamma)$ of $\eta(y,\gamma)$ w.r.t.~$e_t$ is such that, for  $\psi \in \cC(\R^n \times \Gamma_T; \R)$:
\begin{multline}
\int_{\R^n\times \Gamma_T} \psi (y,\gamma) \dd \eta(y,\gamma)\\ 
= \int_{\R^n} \int_{e_t^{-1}(x)} \psi(y,\gamma) \dd\eta_{t,x} (y,\gamma) \dd\mu_t(x).
\end{multline}
Then a mean-vector field is introduced as follows: 
\begin{gather}\label{eq:sp_v}
\widetilde{f}(t,x) := \int_{e_t^{-1}(x)} \dot{\gamma}(t) \dd\eta_{t,x}(y,\gamma).
\end{gather}
The velocity vector \eqref{eq:sp_v} can be understood as a weighted mean of all the velocity vectors $\dot\gamma(t)$ over the curves $\gamma$ passing through point $x$ at time $t$. Note that the convexity of $F(t,x)$ ensures that the mean-velocity \eqref{eq:sp_v} belongs to the set $F(t,x)$. 
We show that the vector field defined in \eqref{eq:sp_v} is equivalent to the vector field defined in \eqref{eq:mean_field} for some appropriate choice of $\omega \in \cP(\R^n)$. To establish this, we introduce a velocity evaluation operator $\mathbf{d}_t : \R^n \times \Gamma_T \to \R^n$, which is a Borel measurable map defined by
\begin{equation}\label{eq:defVelEval}
\mathbf{d}_t(y,\gamma) := \dot{\gamma}(t) {\rm ~with~} \gamma(0) = y .
\end{equation}
Using this mapping, we define:
\begin{equation}\label{eq:pushOmega}
\omega(\cdot|t,x) := {\mathbf{d}_{t}}_{\#} \eta_{t,x} (\cdot).
\end{equation}
\begin{proposition}\label{remark:spDI}
	Let $\mathbf{d}_t$ be the velocity evaluation operator in \eqref{eq:defVelEval}.  Then, tor each $t\in [0,T]$ and $x \in \R^n$, it holds that $\mathbf{d}^{-1}_t(F(t,x)) = e^{-1}_t(x)$. Moreover, for the measure $\omega$ defined in \eqref{eq:pushOmega}, the associated  vector field in \eqref{eq:mean_field} is equal to \eqref{eq:sp_v}.
\end{proposition}
\begin{proof}
	The proof of the first claim follows from the definition, i.e.,
	\begin{gather}
	\mathbf{d}_t^{-1}(F(t,x)) = \{ (\gamma(0),\gamma); \gamma(t) = x \} 	\\
	e_t^{-1}(x) = \{ (\gamma(0),\gamma); \gamma(t) = x \} .
	\end{gather}
	So the two sets are the same.
	Next, we prove the equivalence of the two vector fields \eqref{eq:sp_v} and \eqref{eq:mean_field}. Using the definition of $\overline f_\omega$ in \eqref{eq:mean_field} and the equality in \eqref{eq:pushOmega}, we get
	\begin{gather*}
	\overline{f}_{\omega}(t,x)=  \int\limits_{F(t,x)}  \dot\gamma(t) \; \dd ({\mathbf{d}_{t}}_{\#} \eta_{t,x})(y,\gamma) .
	\end{gather*}
	Under the change of variables in the above equation,
	\begin{gather*}
	\overline{f}_{\omega}(t,x) = \int\limits_{\mathbf{d}_t^{-1}(F(t,x))} \dot{\gamma}(t) \; \dd\eta_{t,x} (y,\gamma).
	\end{gather*}
	Now using $\mathbf{d}^{-1}_t(F(t,x)) = e^{-1}_t(x)$, we get
	\begin{gather*}
	\overline{f}_{\omega}(t,x) = \int\limits_{e_t^{-1}(x)} \dot{\gamma}(t) \; \dd\eta_{t,x} (y,\gamma) =\tilde{f}(t,x)
	\end{gather*}
	for each $t\in [0,T]$ and $x \in \R^n$.
\end{proof}
Thus, the set of trajectories $\Gamma_T^\omega \subset \Gamma_T$ for \eqref{eq:mean_field} and \eqref{eq:sp_v} are the same under the constraint prescribed in \eqref{eq:pushOmega}.
%%%%%%%%%%%%%%%%%%
\subsection{Continuity Equation and its Measure Solution}
We now state the main results of this section concerning the formulation of the continuity equation. In Proposition~\ref{prop:ce}, we show that, for every $\omega(\cdot \vert t,x) \in \cP(F(t,x))$ and every $\eta$ concentrated on $\R^n \times \Gamma_T^\omega$, the image measure $\mu_t = {e_t}_{\#} \eta$ satisfies the  continuity equation driven by $\ol f_\omega(t,x)$ in \eqref{eq:mean_field}. Starting from this equation, in Theorem \ref{thm:sol_le} we discuss the converse statement and characterize all the measure solutions to the derived continuity equation. This characterization of the solutions is especially important as later we propose a numerical method for the  simulation of measure evolution through nonsmooth dynamical systems as the solution of this continuity equation.

%It is known that given absolutely continuous trajectories for any dynamical system, we can associate a measure $\eta \in \P(\R^n \times \Gamma_T)$ (as introduced in Section~ \ref{subsec:sp}) such that the image measure $\mu_t = {e_t}_{\#}\eta$ solves the  continuity equation driven by the vector field describing the dynamical system. Conversely, for each measure solution $\mu_t$ of a continuity equation,  there exists measure $\eta$ (possibly nonunique) such that it is concentrated on absolutely continuous solutions of the underlying dynamical system.

\begin{proposition}\label{prop:ce}
	Consider system~\eqref{eq:di} under Assumption \ref{assume:3} and Assumption \ref{assume:4}. For each $t\in [0,T]$ and $x \in \R^n$, let $\omega(\cdot \vert t,x)$ be a probability measure supported on $F(t,x)$, and let $\overline f_\omega$ and $\Gamma_T^\omega$ be defined as in \eqref{eq:mean_field} and \eqref{eq:defGammaSub}, respectively. Then, for every $\eta \in \cP(\R^n \times \Gamma_T^\omega)$, the measure $\mu_t := {e_t}_{\#} \eta$ satisfies the following continuity equation driven by $\overline{f}_{\omega}(t,x)$,
	\begin{multline} \label{eq:ce_eqn}
	\int_{\R^n} \varphi(T,x) \dd\mu_T(x) - \int_{\R^n} \varphi(0,x) \dd\mu_0(x) = \\ 
	\int\limits_{[0,T]\times \R^n}  \Big[\partial_t \varphi(t,x) 
	+ \nabla_x \varphi(t,x) \cdot \overline{f}_{\omega}(t,x)\Big] ~\dd\mu_t(x)\dd t
	\end{multline}
	for every compactly supported $\varphi \in \cC^1([0,T] \times \R^n; \R)$.% and $\nabla_x$ is the gradient operator.	
\end{proposition}

\begin{proof}
	To derive \eqref{eq:ce_eqn}, we start by proving that the mapping $t \mapsto \int \phi(x) \dd\mu_t(x)$ is absolutely continuous,%
	\footnote{We may consider test functions $\varphi(t,x)=\rho(t)\phi(x)$ which are dense in $\cC^1(\R \times \R^n; \R)$, and then the differentiability of $\int \varphi(t,x) \dd\mu_t(x)$ depends on the absolute continuity of $\int \phi(x)\dd\mu_t(x)$ since
		\[ \frac{\dd}{\dd t}\int_{\R^n} \rho(t) \phi(x) \dd\mu_t(x) = \int_{\R^n} \partial_t(\rho(t))\phi(x) \dd \mu_t(x)+\int_{\R^n} \rho(t)\frac{\dd}{\dd t}\phi(x)\dd\mu_t(x).\]
		So we need to prove that $t\mapsto \int \phi(x)\dd\mu_t(x)$ is absolutely continuous.}
	 for compactly supported $\phi \in \cC^1(\R^n;\R)$. We then use the property of almost everywhere differentiability of absolutely continuous functions to differentiate $\int \phi(x) \dd\mu_t(x)$ w.r.t. time.
	
	\textit{Absolute continuity of $\int \phi \dd\mu_t$ }: Consider the pairwise disjoint intervals $(\ul t_i,\ol t_{i}) \subset [0,T]$, such that $\sum_{i=1}^N (\ol t_i - \ul t_i) < \delta$, for a given $\delta > 0$. 
	Choose $\phi \in \cC^1(\R^n;\R)$, then for any $\gamma \in {\Gamma}_T^\omega$ we have
	\begin{multline}\label{eq:1}
	\sum_{i=1}^N \phi(\gamma(\ol t_i))-\phi(\gamma(\ul t_i)) =\\ \sum_{i=1}^N \int_{(\ul t_i, \ol t_i)} \Big(\nabla_x\phi(\gamma(t)) \Big) \cdot \overline{f}_{\omega}(\gamma(t))\dd t.
	\end{multline}
	Integrating \eqref{eq:1} with $\eta \in \P(\R,\Gamma_T^\omega)$ leads to
	\begin{multline} 
	\sum_{i=1}^N \int_{\R^n \times \Gamma_T^\omega} \Big[\phi(\gamma(\ol t_i))-\phi(\gamma(\ul t_i))\Big] \dd\eta(x,\gamma) \\ = \sum_{i=1}^N \int_{(\ul t_i, \ol t_i)} \int_{\R^n \times \Gamma_T^\omega}\Big(\nabla_x\phi(\gamma(t)) \Big)\cdot \overline{f}_{\omega}(\gamma(t)) \dd\eta(x,\gamma) \dd t .
	\end{multline}
	Now using  $\mu_t = {e_t}_{\#}\eta$  on the left side of the above equation, taking the absolute values on both sides and then using H\"older's inequality we get
	\begin{multline}\label{eq:le2}
	\sum_{i=1}^N\Big|\int_{\R^n}\phi(x) \dd\mu_{\ol t_i}(x)-\int_{\R^n}\phi(x) \dd\mu_{\ul t_i}(x) \Big| \\
	\le ||\nabla_x\phi||_{\infty} \sum_{i=1}^N \int_{(\ul t_i, \ol t_i)} \int_{\R^n \times \Gamma_T^\omega} |\overline{f}_{\omega} (\gamma(t))| \dd\eta(x,\gamma) \dd t . 
	\end{multline}	
	Using the growth bounds on the vector field (Assumption~\ref{assume:4}), we can derive the estimate $\int_{\R^n\times \Gamma_T^\omega} |\overline{f}_{\omega}(\gamma(t))| \dd\eta(x,\gamma)\le K \beta(t)$ for some $K > 0$; refer to Appendix \ref{appendix:1} for details. Substituting this inequality in \eqref{eq:le2}, we get
	\begin{multline*}
	\sum_{i=1}^N \Big |\int\phi(x) \dd\mu_{\ol t_i}(x)-\int\phi(x) \dd\mu_{\ul t_i}(x) \Big| \\ \le K \| \nabla_x\phi \|_{\infty} \sum_{i=1}^N \int_{(\ul t_i,\ol t_i)} \beta(t) \dd t.
	\end{multline*}
	Since $\beta$ is integrable and $\sum_{i=1}^N (\ol t_i- \ul t_i) < \delta$ for an arbitrary $\delta > 0$, the right-hand side can be made arbitrarily small. This proves the absolute continuity of $t \mapsto \int \phi(x) d\mu_t(x)$.
	
	Next, we differentiate $\int \varphi(t,x) d\mu_t(x)$ for any $\varphi(t,x) \in \cC^1_c([0,T] \times \R^m)$ and we obtain the following (refer to Appendix \ref{appendix:2} for details),
	\begin{multline*}
	\int_{\R^n} \varphi(T,x) \dd\mu_T(x) - \int_{\R^n} \varphi(0,x) \dd\mu_0(x) \\= 
	\int_{[0,T]\times \R^n} \Big(\partial_t  \varphi(t,x) +   \nabla_x \varphi(t,x) \cdot \overline{f}_{\omega}(t,x)  \Big) \dd \mu_t(x)\dd t 
	\end{multline*}
	which shows the desired relation. 
\end{proof}
We now rewrite equation \eqref{eq:ce_eqn} in a form which we will use in the optimization problem proposed later in Section~\ref{sec:numerics}. Substituting the expression for $\overline{f}_{\omega}(t,x)$ from \eqref{eq:mean_field} in \eqref{eq:ce_eqn}, we get
\begin{multline}\label{eq:ce_ns}
\int_{\R^n} \varphi(T,x) \dd \mu_T(x) - \int_{\R^n} \varphi(0,x) \dd \mu_0(x) = \\ 
\int_{[0,T]\times \R^n} \Big(\partial_t  \varphi(t,x) +  \nabla_x \varphi(t,x) \cdot \int_{F(t,x)} v \; \dd\omega(v|t,x) \Big)\dd \mu_t(x) \dd t .
\end{multline} 
Rearranging the terms in the above equations and defining $\dd \hat{\mu}(t,x,v)=\dd\omega(v|t,x) \dd \mu_t(x)dt$, we get
\begin{multline}\label{eq:ce_nsd}
\int_{\R^n} \varphi(T,x) \dd \mu_T - \int_{\R^n} \varphi(0,x) \dd \mu_0 = \\
\int\limits_{[0,T]\times \R^n} \int\limits_{ F(t,x)} \Big[\partial_t \varphi(t,x) 
+ \nabla_x \varphi(t,x) \cdot v\Big] ~\dd\hat{\mu}(t,x,v).
\end{multline}
\begin{comment}
\begin{remark}
One can arrive at \eqref{eq:ce_nsd} by differentiating any $\varphi\in \cC^1([0,T]\times \R^n)$ along the trajectories $\gamma \in \Gamma_T$ and using the definition of occupation measures \eqref{eq:occ_meas_def}. Thus, above dynamical equation \eqref{eq:ce_nsd} is satisfied by the ocupation measures defined on the solution trajectories $\Gamma_T$ of $\dot{x}\in F(t,x)$.
\end{remark}
\end{comment}
Equation \eqref{eq:ce_nsd} will be the starting point for the next result as we will characterize all possible solutions $\hat{\mu}$ to it.
The solutions $\hat{\mu}$ will determine the vector field \eqref{eq:mean_field} (by defining $\omega$) and the solutions to the continuity equation driven by this vector field will lead to a measure concentrated on the trajectories of the derived vector field. 
\begin{comment}
\begin{remark}
$\overline\mu(t,x)$ is an occupation measure concentrated on the family of trajectories of system \eqref{eq:nsd}. to see this we integrate the $\overline{\mu}(t,x)$ on $A\times B \subset [0,T]\times \R^n$,
\begin{align*}
\overline{\mu}(A,B) = \int \I_A(t) \I_B(x) d\mu_t dt
\end{align*}
Using the superposition principle, there exists $\sigma \in \P(\Gamma_T)$ concentrated on the absolutely continuous solutions to the ODE system defined by the vector field \eqref{eq:di}. 
\begin{align*}
\int \phi(x) d\mu_t = \int_{\Gamma} \phi(x(t)) \sigma(x(\cdot)) 
\end{align*}
Substituting $\phi(\cdot) = \I_B$  in the above definition and then applying Fubini's theorem we can prove that,
\begin{align*}
\overline{\mu}(A,B) = \int_{\Gamma_T} \int_a^b \I_A(t) \I_B(x(t)) dt \sigma(x(\cdot)) 
\end{align*}
This shows that $\overline{\mu}(A,B)$ is an occupation measure conentrated on the family of trajectories of system \eqref{eq:nsd}.
\end{remark}
\end{comment}
\begin{theorem}\label{thm:sol_le}
	Consider system~\eqref{eq:di} under Assumption \ref{assume:3} and Assumption \ref{assume:4}.  Any measure $\hat{\mu}$ that solves the continuity equation \eqref{eq:ce_nsd} is of the form 
	\begin{equation}\label{eq:disIntegMeasure}
	\dd\hat{\mu}(t,x,v) = d\omega(v|t,x) d\mu_t(x) \dd t
	\end{equation}
	where  $\omega(\cdot|t,x) \in \P(F(t,x))$ and $\mu_t$ solves \eqref{eq:ce_eqn}. %with $\overline{f}_{\omega}(t,x)$ defined in \eqref{eq:mean_field}. %The measure $\mu_t$ defines $\eta$ concentrated on set of trajectories of \eqref{eq:mean_field} $\overline{\Gamma}\subset\Gamma$.
\end{theorem}
\begin{proof}
	\begin{comment}
	The proof will be carried in three steps:
	\begin{itemize}
	\item Show that the measure $d\hat{\mu}(t,x,\dot{x})$ can be decomposed as a product of $d\overline{\mu}(t,x)$ and $dw(\dot{x}|t,x)$. 
	\item Further, $\overline{\mu}$ can be decomposed as $d\overline{\mu}(t,x) = d\mu_t(x)dt$.
	\item Conclude that $\mu_t$ will concentrated on the family of  solution trajectories of \eqref{eq:di}.
	\end{itemize}
	\end{comment}
	In Euclidean space $\R^n$, we can use the disintegration theorem \citep[Corollary 10.4.13]{Bogachev_2007} to write $\dd\hat{\mu}(t,x,v)= \dd\omega(v|t,x)d\overline{\mu}(t,x)$ where $\omega(\cdot|t,x) \in \P(F(t,x))$.  Using this we can rewrite 
	\begin{multline}
	\int_{\R^n} \varphi(T,x) \dd\mu_T - \int_{\R^n} \varphi(0,x) \dd\mu_0 = \\
	\int\limits_{[0,T]\times \R^n} \int\limits_{ F(t,x)} \Big[\partial_t \varphi(t,x) 
	+ \nabla_x \varphi(t,x) \cdot \zeta\Big] ~\dd\omega(\zeta|t,x) \dd\overline\mu(t,x) .
	\end{multline}
	Rearranging the terms results in
	\begin{multline}\label{eq:thm1_main_eqn}
	\int_{\R^n} \varphi(T,x) \dd\mu_T(x) - \int_{\R^n} \varphi(0,x) \dd\mu_0(x) =\\ 
	\int_{[0,T]\times \R^n} 	\Big(\partial_t  \varphi(t,x) +   \nabla_x \varphi(t,x) \cdot \overline{f}_{\omega}(t,x)  \Big)\dd\overline{\mu}(t,x) 
	\end{multline} 
	where $\overline{f}_{\omega}$ is defined as
	\begin{align}
	\overline{f}_{\omega}(t,x) = \int_{F(t,x)} \zeta \dd \omega(\zeta|t,x) ~ \mathrm{for}~ \omega(\cdot|t,x) \in \P(F(t,x))
	\end{align}
	and spans the set $F(t,x)$ as we have assumed $F(t,x)$ to be convex for every $t \in [0,T]$. 
	
	\textit{Decomposition of $\overline{\mu}$:}
	Next we show that the marginal of $\overline{\mu}(t,x)$ w.r.t. time is  a Lebesgue measure. This can be shown by taking $\varphi(t,x) = t^k$ for some $k \ge 0$ in \eqref{eq:thm1_main_eqn}, then we get
	\begin{gather}
	\mu_T(\R^n) T^k - \int_{\R^n} t^k \dd\mu_0=   \int_{[0,T]\times \R^n} kt^{k-1}\dd\overline{\mu}(t,x)
	\end{gather}
	where taking $k=0$ gives $\mu_T(\R^n) = \mu_0(\R^n)$ and for $k\ge 1$ results in  $ \frac{\mu_T(\R^n) T^k}{k} = \int_{[0,T]\times \R^n} t^{k-1}\dd\overline{\mu}(t,x)$.
	So, up to scaling we can write $\overline{\mu}(\dd t,\dd x) = \mu_t(\dd x) \dd t$. Substituting this we arrive at the following continuity equation
	\begin{multline}
	\int_{\R^n} \phi(T,x) \dd\mu_T - \int_{\R^n} \phi(0,x) \dd\mu_0 
	=\int\limits_{[0,T]\times \R^n}  \Big[\partial_t \phi(t,x) \\
	+ \nabla_x \phi(t,x) \cdot \overline{f}_{\omega}(t,x)\Big] \dd\mu_t(x)\dd t.
	\end{multline}
	Now using the results in \citep{Amb08}, the only solutions to the continuity equation have a representation in terms of measure $\eta \in \P(\R^n\times \Gamma)$ (as defined in Section \ref{subsec:sp}) as $\mu_t = {e_t}_{\#}\eta$ where $\eta$ are concentrated on the solution trajectories to system $\dot{x}(t)= \overline{f}_{\omega}(t,x(t))$. 
\end{proof}

\subsection{Measure Evolution for Constrained Systems}\label{ce_sweep}
Using the equation \eqref{eq:ce_nsd}, we arrive at the continuity equation for the system defined in \eqref{eq:nsd}, i.e., the equation
\begin{multline}\label{eq:le_nsd}
\int_{S(T)} \phi(T,x) d\mu_T(x) - \int_{S(0)} \phi(0,x) d\mu_0(x) 
\\=\int\limits_{[0,T]\times S(t)} \int\limits_{ f(t,x) - \cN_{S(t)}(x)} \Big[\partial_t \phi(t,x) 
+ \nabla_x \phi(t,x) \cdot \zeta\Big] ~d\hat{\mu}(t,x,\zeta)
\end{multline}
holds for every compactly supported $\phi \in \cC^1([0,T]\times \R^n; \R)$.
%The proof of uniqueness of the solutions to the transport equation is deduced by posing existence of $\cC^1(\R^m)$ solution to a boundary value problem\citep{Car10}. But this existence of solution to boundary value problem is not guaranteed for vector fields which are not Lipschitz.  So, we take a different route based on the superposition principle. Now given an initial condition $x_0\in\R^n$, the solutions to \eqref{eq:nsd} is unique \citep{SOUAIBY2023110836}. 
%
Using Theorem~\ref{thm:sol_le}, it follows that the measure solutions to the continuity equation have a representation~\eqref{eq:disIntegMeasure}. In this decomposition, $\omega(\cdot \vert t,x)$ represents the selection from the set $f(t,x) - \cN_{S(t)}(x)$ so that the resulting trajectories evolve within the set $S(t)$. Corresponding to such selections, the solutions to ODE \eqref{eq:nsd} are unique with $x(0) \in S(0)$, see for example \citep[Section~5]{BroTan20}. Consequently, we have that $\eta = \delta_{X_t(x_0)}$ and due to Proposition~\ref{prop:ce}, the measure $\mu_t = {e_t}_{\#}\eta$ corresponds to 
\begin{gather} \label{eq:ce}
\mu_t := {X_{t}}_{\#}\mu_0
\end{gather}
where $X_t$ denotes the flow map associated with system~\eqref{eq:nsd}.

\section{Functional Regularization}\label{sec:fun_reg}

As noted earlier, the right-hand side of system~\eqref{eq:nsd} is possibly discontinuous, and this introduces complexity in writing the transport equation for measures. The second approach that we propose relies on working with Lipschitz continuous approximations of the right-hand side of \eqref{eq:nsd} to generate a sequence of approximate solutions $\{x^{\lambda}\}_{\lambda > 0}$ parameterized by $\lambda > 0$. In particular, we work with the so-called {\it Moreau-Yosida regularization}, which for system~\eqref{eq:nsd} takes the following form:
\begin{equation}\label{eq:moreau_yosida}
\begin{aligned}
\dot{x}^\lambda(t) &= g_t^{\lambda}(x^\lambda(t)) \\
& := f(t,x^\lambda(t)) - \frac{1}{\lambda}(x^\lambda(t)-{\rm proj}(x^\lambda(t),S(t)))
\end{aligned}
\end{equation}
where we take $x^\lambda (0) = x(0) \in S(0)$, and ${\rm proj}(x^\lambda(t),S(t))$ refers to the projection of the vector $x^\lambda(t)$ onto the set $S(t)$ with respect to the Euclidean distance.
It is well known that the solution curve $\mul$ of the continuity equation with the Lipschitz regular vector field \eqref{eq:moreau_yosida} satisfies the following pushforward relationship
\begin{gather}\label{eq:fun_reg_repForm}
\mu^{\lambda}_t = {X^{\lambda}_t}_{\#} \mu_0
\end{gather}
where $X^{\lambda}_t(x_0)\coloneqq x^{\lambda}(t,x_0)$ is the flow map associated with \eqref{eq:moreau_yosida}. In \citep{SOUAIBY2023110836}, the authors show that the solution $x^{\lambda}(t)$ solving \eqref{eq:moreau_yosida} converges uniformly to the solution $x(t)$ of \eqref{eq:nsd} when $x^\lambda(0) = x(0) \in S(0)$, and that the measures $\mul$ converge in weak star topology to the measure solutions $\mu_t = {X_t}_{\#}\mu_0$, with $X_t$ being the flow map of \eqref{eq:nsd}. In this section, we provide quantitative bounds on the Wasserstein distance between measures $\mu_t$ and $\mul$.% similar to \eqref{eq:w2_bound}.

%\begin{comment}
%{\color{red} 
\begin{rem}
As a parallel to the regularization technique presented here, we find an approach based on mollification in \citep{Amb08} to study the evolution of measures for nonsmooth dynamical systems. 
Such mollification is carried out by using a convolution kernel $\psi : \R^n \to [0,\infty)$ with the properties that  $\psi(x)$ is bounded, measurable with $\psi(x)=\psi(-x)$, $\int \psi(x) dx = 1$. Let $\psi_{\epsilon} := \frac{1}{\epsilon^n}\psi(\frac{x}{\epsilon})$, and the corresponding convolution with a measure $\mu$ as $(\mu * \psi_{\epsilon})(x) := \int \psi_{\epsilon}(x-y)d \mu(y)$. For $\mu^{\epsilon} \coloneqq \mu * \psi_{\epsilon}$, it can be shown that
$W_2(\mu^{\epsilon},\mu) \le \epsilon \int{\vert \psi (x) \vert^2 \dd x}$.
In \citep{Amb08}, a similar mollification technique was used and the narrow convergence\footnote{Note: Family of measures $\mu_n$ converges narrowly to measure $\mu$ if $\lim_{n\to \infty}|\int fd\mu_n - \int fd\mu| \to 0$ for a bounded $f\in \cC(\Omega;\R)$, where $\Omega$ is any Polish space. Note that the definition is different from weak* convergence where the convergence is defined w.r.t. compactly supported continuous functions \citep{AmbrGigl05}. When the underlying space $\Omega$ is compact both the notions of convergence coincide.} of the measures was proven by working with a smooth vector field $g^{\epsilon} \coloneqq \frac{(g \mu)*\psi_\epsilon}{\mu^{\epsilon}}$ with the corresponding continuity equation $\partial_t\mu^{\epsilon}_t + \nabla\cdot{(g^{\epsilon}_t \mu^{\epsilon}_t)} = 0$.
\end{rem}
%}
%\end{comment}

\begin{theorem}\label{thm:fun_reg}
	Let $\mul\in\P(\R^n)$ be defined as in~\eqref{eq:fun_reg_repForm}, and let $\mu_t = {X_t}_{\#}\mu_0$, with $\mu_0 \in \cP(S(0))$ and $X_t$ being the flow map of \eqref{eq:nsd}. Then, the $W_1$ distance between $\mu_t$ and $\mul$ satisfies the following bound:
	\begin{gather} \label{eq:W1_bound}
	W_1(\mu _t,\mul)  \le   C_1  \sqrt{\frac{L_f \lambda(e^{L_f t}-1)}{2}} 
	\int_{S(0)}  \vert x_0 \vert d\mu_0(x_0)
	\end{gather}
	where $C_1 = L_f(1 + \kappa) + L_S$  and $\kappa := (e^{2L_f T}-1) \frac{2 L_f +L_s}{2 L_f}$.
\end{theorem} 
\begin{rem}
	Under the assumption that  $\supp(\mu_0)$ is compact, $\supp(\mu_t)$ and $\supp(\mul)$ will be compact as these are the push-forwards of Lipschitz continuous operators. This results in $W_2 (\mul, \mu_t) \le C_2 W_1(\mul,\mu_t)$ for some $C _2 > 0$ and hence a qualitatively similar bound holds for $W_2$ metric as the one indicated in \eqref{eq:W1_bound}.
\end{rem}
%	Thus, \[W_2 (\mul, \mu_t) \le C W_1(\mul,\mu_t)\] for some $C\ge 0$ and substituting the expression for $W_1(\mul,\mu_t)$ we get the desired bound on $W_2(\mul, \mu_t)$.
%Furthermore, in the case where $\supp(\mu_0)$ is bounded, \[W_2(\mul,\mu_t) \le C \sqrt{\frac{\lambda}{2}\frac{1-e^{L_f t}}{L_f}} 
%	\int |L_f(1+\kappa+|x_0|)+L_s| d\mu_0(x_0)\]
%	for some $C\ge0$.

We will use the following lemma in the proof of Theorem~\ref{thm:fun_reg}.
\begin{lemma} \label{lem:fun_reg}
	Let $\xl(t)$ be the solution to \eqref{eq:moreau_yosida} and $x(t)$ be the solution to \eqref{eq:nsd} with $x^\lambda (0) = x(0) \in S(0)$. Then, for each $t \in [0,T]$, it holds that
	\[
	|x^{\lambda}(t)-x(t)| \le  (L_f(1+\kappa+|x_0|)+L_s)  \sqrt{\frac{\lambda(e^{2L_f t}-1)}{2 L_f}}
	\]
	where $\kappa = (e^{2L_f T}-1) \frac{2 L_f +L_s}{2 L_f}$.		
\end{lemma}

\begin{proof}
	Let $x^{\lambda}(\cdot)$ and $x^{\nu}(\cdot)$ be solutions to \eqref{eq:moreau_yosida} corresponding to the regularization parameters $\lambda$ and $\nu$, respectively. Then
	\begin{multline*}
	\frac{1}{2}\frac{\dd}{\dd t}(|\xl(t)-\xn(t)|^2) = \langle \xl(t)-\xn(t),\dot{x}^{\lambda}(t)-\dot{x}^{\nu}(t)\rangle\\ 
	= \langle \xl(t)-\xn(t),f(t,\xl(t)) -\frac{1}{\lambda}(\xl(t)-{\rm proj}(\xl(t),S(t))) \\ -f(t,\xn(t))-\frac{1}{\nu}(\xn(t)-{\rm proj}(\xn(t),S(t)))\rangle. \\
	\end{multline*}
	Using the Lipschitz property of $f(t,x)$ and the Cauchy-Schwarz inequality for the terms involving  $\langle \xl(t)-\xn(t), f(t,\xl(t)) - f(t,\xn(t))\rangle$,  we obtain
	\begin{multline}\label{eq:fun_reg_1}
	\frac{1}{2} \frac{\dd}{\dd t}(|\xl(t)-\xn(t)|^2)\le L_f| \xl(t) -\xn(t)|^2 - \Big\langle \xl(t)-\xn(t), \\ -\frac{1}{\lambda}(\xl(t)-{\rm proj}(\xl(t),S(t)))  +\frac{1}{\nu}(\xn(t)-{\rm proj}(\xn(t),S(t)))\Big\rangle .
	\end{multline}
	We rewrite $\xl(t)-\xn(t) = \xl(t)-\p(\xl(t),S(t))+(\p(\xl(t),S(t))-\p(\xn(t),S(t))) -( \xn(t) - \p(\xn(t),S(t)))$ and substitute it in  \eqref{eq:fun_reg_1} to obtain
	\begin{multline}\label{eq:fun_reg_2}
	\frac{1}{2} \frac{\dd}{\dd t}(|\xl(t)-\xn(t)|^2)\le L_f|\xl(t)-\xn(t)|^2 + \\
	\Big\langle \xl(t) \pm \p(\xl(t),S(t)) - \xn(t)  \pm \p(\xn(t),S(t)), \\- \frac{1}{\lambda}(\xl(t)-\p(\xl(t),S(t)))
	+\frac{1}{\nu}(\xn(t)-{\rm proj}(\xn(t),S(t))))\Big\rangle .
	\end{multline}
	For notational convenience, we denote  $\yl=\frac{1}{\lambda}(\xl(t)-\p(\xl(t),S(t)))$, $\yn=\frac{1}{\nu}(\xn(t)-{\rm proj}(\xn(t),S(t)))$ and substitute these in \eqref{eq:fun_reg_2} to obtain
	\begin{multline}\label{eq:fun_reg_3}
	\frac{1}{2} \frac{\dd}{\dd t}(|\xl(t)-\xn(t)|^2) \le L_f|\xl(t)-\xn(t)|^2 + \\
	\Big\langle \lambda \yl +(\p(\xl(t),S(t))-\p(\xn(t),S(t)))  - \nu \yn,\\ - \yl+ \yn\Big\rangle.
	\end{multline}
	It is known that $\yl$ and $\yn$ satisfy the monotonicity property, i.e.
	\begin{gather}\label{eq:fun_reg_4}
	\langle -\yl + \yn, (\p(\xl(t),S(t))-\p(\xn(t),S(t))) \rangle \le 0
	\end{gather}
	since $\yl$ and $\yn$ are Moreau-Yosida regularizations of the $\cN_{S}(x)$ operator.
	We substitute \eqref{eq:fun_reg_4} in \eqref{eq:fun_reg_3} and get 
	\begin{multline}\label{eq:fun_reg_5}
	\frac{1}{2}\frac{\dd}{\dd t}(|\xl(t)-\xn(t)|^2)\le L_f|\xl(t)-\xn(t)|^2 \\+ 
	\langle \lambda \yl - \nu \yn, - \yl+ \yn\rangle.
	\end{multline}
	Using the Cauchy-Schwartz inequality for the second term on the right-hand side of \eqref{eq:fun_reg_5}, we get
	\begin{multline}
	\frac{1}{2}\frac{\dd}{\dd t}(|\xl(t)-\xn(t)|^2)\le L_f|\xl(t)-\xn(t)|^2 
	-\lambda |\yl|^2 - \\\nu |\yn|^2 
	+ (\lambda+\nu)|\yl| \cdot |\yn|.
	\end{multline}
	Next we use Young's inequality for the term $|\yl| \cdot |\yn|$ to obtain
	\begin{multline}\label{eq:fun_reg_6}
	\frac{1}{2} \frac{\dd}{\dd t}(|\xl(t)-\xn(t)|^2)\le L_f|\xl(t)-\xn(t)|^2 
	-\lambda |\yl|^2 -\\ \nu |\yn|^2   +\frac{(\lambda+\nu)}{2}(|\yl|^2+ |\yn|^2) \\
	\le L_f |\xl-\xn(t)|^2 
	+\frac{\nu}{2} |\yl|^2 + \frac{\lambda}{2}|\yn|^2 .
	\end{multline}
	In \citep{SOUAIBY2023110836}, uniform bounds on $|\xl(t)|$ were obtained, which lead to $|Y_{\lambda}(\xl(t))| \le L_f (1+\kappa +e^{2L_f T}|\xl(0)|) +L_s$, where $\kappa = (e^{2L_f T}-1) \frac{2 L_f +L_s}{2 L_f}$ and a similar bound on $\yn$ (see Appendix~\ref{appendix:3}). We use this bound in \eqref{eq:fun_reg_6} and obtain
	\begin{multline}\label{eq:fun_reg_7}
	\frac{1}{2}\frac{\dd}{\dd t}(|\xl(t)-\xn(t)|^2)  \le L_f|\xl(t)-\xn(t)|^2 \\+ \frac{(\nu + \lambda)}{2}|L_f(1+\kappa+e^{2L_f T}|\xl(0)|)+L_s|^2
	\end{multline}
	where we have used the fact that $\xl(0) = \xn(0) = x(0)$. Now, applying  Gronwall's lemma, we get
	\begin{multline}
	|\xl(t)-x^{\nu}(t)|^2  \\
	\le |L_f(1+\kappa+e^{2L_f T}|\xl(0)|)+L_s|^2 (\nu + \lambda)\frac{e^{2L_f t}-1}{2L_f}
	\end{multline}
	where the term involving $|\xl(0)-\xn(0)|$ in the right-hand side is zero. Next, we use the fact that $\lim\limits_{\nu\to 0} \xn (t) = x(t)$ \citep{SOUAIBY2023110836} holds pointwise and $\xn(0)=x(0)=x_0$ to obtain the desired bound.
	%	\begin{gather}
	%	 |x^{\lambda}(t)-x(t)| \le  |L_f(1+\kappa+e^{2L_f T}|x_0|)+L_s|  \sqrt{\frac{\lambda}{2}\frac{|1-e^{L_f t}|}{L_f}}.
	%	\end{gather}	
\end{proof}
\begin{proof}[Proof of Theorem \ref{thm:fun_reg}]  
	To get a bound on the distance between $\mu_t$ and $\mu^{\lambda}_t$, we use the dual characterization of the $W_1$ distance \citep{filipoOT}: 
	\begin{gather}\label{eq:defDualW1}
	W_1(\mul,\mu_t) = \sup_{\substack{\phi \in \cC(\Omega;\R),\\ ||\phi||_{\rm Lip}\le 1}} \int \phi \dd (\mul-\mu_t)
	\end{gather}
	where $\| \phi \|_{\rm Lip}$ denotes the Lipschitz modulus of $\phi$ and $\Omega$ is a measurable set containing all $S(t)$ for $t \in [0,T]$.   We use the representation formula for $\mul$ and $\mu_t$ to obtain,
	\begin{multline*}
	\sup_{\substack{\phi\in C(\Omega;\R)\\ ||\phi||_{\rm Lip}\le 1}} \int \phi \, \dd  (\mu^{\lambda}_t-\mu_t) \\
	=\sup_{\substack{\phi\in C(\Omega;\R)\\ ||\phi||_{\rm Lip}\le 1}} \int \phi(x^{\lambda}(t,x_0)) \dd\mu_0(x_0) - \int \phi(x^{\lambda}(t,x_0)|_{\lambda = 0})\dd\mu_0(x_0)
	\end{multline*}
	where $x^{\lambda}(t,x_0)|_{\lambda = 0} \coloneqq \lim_{\lambda \to 0} x^{\lambda}(t,x_0)$. 
	Using the first order Taylor expansion of $\phi(x^{\lambda}(t,x_0))$ w.r.t. $\lambda$, for each $\phi \in \cC(\Omega;\R) {\rm ~s.t.~} ||\phi||_{\rm Lip}\le 1~$,  we get,
	\begin{align}\nonumber
	\int& \phi(x^{\lambda}(t,x_0)) \dd\mu_0 - \int \phi(x^{\lambda}(t,x_0)|_{\lambda = 0})\dd\mu_0 \le \\ \nonumber
	\int& \Big[\phi(x^{\lambda}(t,x_0)|_{\lambda = 0}) + \nabla_x \phi(x^{\lambda}(t,x_0))|_{\lambda = 0}\cdot (x^{\lambda}(t,x_0)- \\ \nonumber &\hspace{24mm}x(t,x_0))\Big ] \dd \mu_0  - \int \phi(x(t,x_0)) \dd \mu_0\\ \nonumber
	=& \int \nabla_x \phi(x^{\lambda}(t,x_0))|_{\lambda = 0}(x^{\lambda}(t,x_0)-  x(t,x_0)) \dd\mu_0 \\ 
	\le&\int |\nabla_x \phi(x^{\lambda}(t,x_0))|_{\lambda = 0}|~ | (x^{\lambda}(t,x_0)-  x(t,x_0))| \dd\mu_0.
	\end{align}
	Using the fact that $\phi$ is of Lipschitz constant $1$, the above equation reduces to 
	\begin{gather}\label{eq:fun_reg_8}
	W_1(\mul,\mu_t) \le \int_{S(0)} |(x^{\lambda}(t,x_0)-  x(t,x_0))| d\mu_0(x_0).
	\end{gather}
	Using Lemma \ref{lem:fun_reg} in \eqref{eq:fun_reg_8}, we get the inequality in \eqref{eq:W1_bound}.
\end{proof}

%The above $W_2$ estimates between $\mul$  and $\mu_t$ gives a quantitative estimate of the rate of convergence. This bound establishes a parallel between the Moreau-Yosida based regularization and mollification of the measures \eqref{eq:mollify}.

\section{Time Discretization and Optimal Transport}\label{sec:td_pde}

In this section, we provide a construction of solutions to the continuity equation for \eqref{eq:nsd} using a time discretization scheme. %The method presented here relies on  tools presented in \citep{filipoOT} and \cite{marino2016}. %For the notations related to optimal transport the reader is referred to Section \ref{subsec:ot}. 
In what follows, we consider $\Omega \subset \R^n$ to be a compact set which contains $S(t)$, for all $t\in[0,T]$.

\subsection{Absolutely Continuous Curves of Measures}
In what follows, we consider the space of probability measures, with bounded second moment and equipped with the $W_2$ metric, which is denoted by $(\P_2(\Omega), W_2)$, see \citep{AmbrGigl05} for details. This metric space is called the Wasserstein space and it will be denoted as $\cW_2(\Omega)$.
We say that a curve $[0,T] \ni t \mapsto \mu_t \in ( \P_2(\Omega), W_2) $ is absolutely continuous if there exists $m(t) \in L^2([0,T])$ such that
\begin{gather}\label{eq:AC}
\lim_{h \to 0} \frac{W_2(\mu_t,\mu_{t+h})}{h} \le |m|(t)
\end{gather} 
holds for almost every $t \in [0,T]$. From \citep[Theorem ~1.1.2]{AmbrGigl05}, the metric derivative $\mu'(t)$ is defined such that $| \mu'|(t) \le \vert m \vert(t)$ holds for all functions $m \in L^2([0,T])$ satisfying \eqref{eq:AC}.

Another important characterization of absolutely continuous curves  in  the $\cW_2$ space is as follows. A curve $\mu_t : [0,T]\to \cW_2(\Omega)$ is absolutely continuous if and only if there exists a vector field $v_t$ with $||v_t||_{L_2(\mu_t)} = |\mu'|(t)$ and $\mu$ satisfies the continuity equation driven by the drift term $v_t$.

\subsection{Construction of Curves in Wasserstein Space}
Next we propose a construction of curves in $\cW_2(\Omega)$ through an interpolation between measures defined at discrete time instants using a discretization of the nonsmooth dynamical system \eqref{eq:nsd}. Consider a partition $\{ 0=t_0,t_1,... t_i,... t_N=T\}$ of time interval $[0,T]$ such that $t_{k+1}-t_k =\tau$. For  a fixed value of $\tau$, we now define the measures $\{\rhoe\}_{k \in \N}$ at time instants $t_k$ in a recursive manner. To do so, let $S_k := S(k\tau)$ which is a closed convex set under Assumption~\ref{assume:2}. We denote by $P_{S_k}$ the projection mapping onto the set $S_k$, and we consider the mapping $G^k : \R^n \to \R^n$, defined as
\[
x \mapsto G^k (x) := {P_{S_{k+1}}} \circ (\tau f_k(x) + x)
\] 
with $f_k(x) := f(t_k,x)$. 
The successor of $\mu_k^\tau$ is now defined as its push-forward under the mapping $G^k$ as follows:
\begin{equation}\label{eq:defSuccmu}
\mu_{k+1}^\tau := G^k_{\#} \rhoe = \Big[{P_{S_{k+1}}} \circ(\tau f_k(\cdot) + \id)\Big]_{\#} \rhoe.
\end{equation}
Similarly, for each $x \in S(k\tau)$, the velocity vector at time $(k+1)\tau$ is defined as 
\begin{equation*}
v^{\tau}_{k+1}(x) \coloneqq \frac{G^k(x)-x}{\tau} = \frac{{P_{S_{k+1}}} \circ (\tau f_k(\cdot)+ \id)(x)-x}{\tau} .
\end{equation*}
Next we consider the following two different interpolation curves which will serve different purposes:

(1) Geodesic interpolation between $\mu^{\tau}_k$ and $\mu^{\tau}_{k+1}$ over the interval $(t_k, t_{k+1}]$ by defining the transport maps
\[
G_t := \Big(\frac{t - k\tau}{\tau}G^* +\frac{(k+1)\tau -t }{\tau} \id \Big)
\] 
for each $t \in (k\tau,(k+1)\tau]$, for some optimal transport map $G^*$ between $\rhoe$ and $\rhoo$ and letting
\begin{multline}\label{eq:geo_trnsptmp}
\mu^{\tau}_t = {G_t}_{\#} \rhoe = \Big(\frac{k\tau -t }{\tau}G^* +\frac{(k+1)\tau -t }{\tau} \id \Big)_{\#} \rhoe.
\end{multline}
The map $G_t$ is injective and the proof is based on the c-cyclical monotonicity property of optimal transport maps \citep{filipoOT}. The interpolation of velocity vector is defined as 
\begin{gather}\label{eq:defVelTransport}
v^{\tau}_t(x) := v^{\tau}_{k+1} \circ (G_t)^{-1}(x) {\rm ~for}~~ t\in (k\tau,(k+1)\tau].
\end{gather}
The $L^2$ norm of velocity $v^{\tau}_t$ satisfies the following relation
\begin{multline}\label{eq:vel_wass_rel}
||v^{\tau}_t||_{L^2(\mu^{\tau}_t)} = \frac{W_2(\rhoe,\rhoo)}{\tau} = |(\mu^{\tau})'|(t)  \\{\rm for ~all~} t\in (k\tau, (k+1)\tau].
\end{multline}
We further define the momentum vector as 
\begin{gather}\label{eq:geo_mom}
E^{\tau}_{t} := v^{\tau}_{t} \mu^{\tau}_{t}
\end{gather}
and it satisfies $\partial_t \rhot + \nabla \cdot (E^{\tau}_t) = 0$.

(2) Piecewise constant interpolation curve such that 
\begin{align}\label{eq:const_trnsptmp}
\hmut_t &= \rhoo, \\
\hvt_t &= v^{\tau}_{k+1} \quad {\rm ~for}~~ t\in (k\tau,(k+1)\tau] \label{eq:const_vmp}
\end{align}
We also define the corresponding momentum vector as 
\begin{gather}\label{eq:pw_mom}
\hat{E}^{\tau}_t := \hvt_{t} \hmut_t.
\end{gather}
We will use piecewise constant interpolation to show that the limit velocity belongs to \eqref{eq:nsd}.

\subsection{Convergence Result}
Next we illustrate the convergence of the constructed curves $\rhot$, $\hmut_t$ to $\mu_t$  which is solution to the continuity equation associated with \eqref{eq:nsd}.

The results presented here align with those in \citep{marino2016} in the absence of drift term, i.e. when $f=0$. In such methods, the essential technical difficulty is in obtaining a priori estimates on the discretized trajectories which help establish the convergence results. The estimates like Wasserstein distance between the trajectories at two time instants depend on the bounds which are explicit for the sweeping process with drift term. % and they are substantially different for the drift free sweeping process.
\begin{theorem} \label{th:OT_CE}
	Consider system~\eqref{eq:nsd} under Assumption \ref{assume:1} and Assumption \ref{assume:2}, with $\mu_0 \in \P(S(0))$. For $\tau > 0$, and $t \in [0,T]$, let $\rhot$ and $v^{\tau}_{t}$ be defined as in \eqref{eq:geo_trnsptmp} and \eqref{eq:defVelTransport}, respectively. Then, as $\tau \to 0$, we get the following two convergence results:
	\begin{itemize}
		\item The measures $\rhot$ in \eqref{eq:geo_trnsptmp} and $\hat{\mu}^{\tau}_t$ in \eqref{eq:const_trnsptmp} converge uniformly in the $W_2$ metric to $\mu_t = {X_t} _{\#} \mu_0$, where ${X_t}$ is the flow map associated with \eqref{eq:nsd}; %The curve $\hat{\mu}^{\tau}_t$ \eqref{eq:const_trnsptmp} converges uniformly to $\mu_t$ in $W_2$ metric.
		
		\item  Momentum vectors  $E^{\tau}_{t} = v^{\tau}_t \rhot$ (defined in \eqref{eq:geo_mom}) and $\hat{E}^{\tau}_t = \hat{v}^{\tau}_t \hat{\mu}^{\tau}_t$ (defined in \eqref{eq:pw_mom}) converge to $E_t := v_t \mu_t$ in the weak star sense. Moreover, the velocity $v_{t}$ is such that $v_t \in f(t,x)-\cN_{S(t)}(x)$.
	\end{itemize}
\end{theorem}
\begin{proof}
	We split the proof into four parts: (1) Proof of convergence of $\rhot$ to $\mu_t$ and $\hat{\mu}^{\tau}_t\to \mu_t$;  (2)  Proof of convergence of $E^{\tau}_t$ and $\hat{E}^{\tau}_t$ to $E_t$; (3) Absolute continuity of $E_t$ with respect to $\mu_t$ such that $E_t = v_t \mu_t$; and (4) Convergence of $v_t^\tau$ to $v_t$ with the property that $v_t \in f(t,x) - \cN_{S(t)}(x)$. The first of these four items follows from the following result:
	\begin{lemma}\label{lemma:w2_bound_mu}
		Let $\rhoe$ and $\rhoo$ be defined as in \eqref{eq:defSuccmu} and \eqref{eq:geo_trnsptmp}. Then, it holds that
		\begin{gather}\label{eq:w2_bounds_4}
		W_2(\rhoo,\rhoe)\le  \tau( L_f C_{\max}+L_s )
		\end{gather}
		where the constants $L_f, L_s$ are defined in Assumption \ref{assume:1} and $C_{\max}$ is a constant that captures the uniform bound on $|x_k|$ at time instant $t_k$, independently of $k \in \N$.
	\end{lemma}
	\begin{proof}[Proof of Lemma~\ref{lemma:w2_bound_mu}]
		The mapping $G^k(x)$  defines a feasible transport map between $\rhoe$ and $\rhoo$. As the Wasserstein distance is defined to be the infimum over all feasible transport maps, we have
		\begin{gather}
		W_2^2(\rhoo,\rhoe) \le \int_{S_k} |{P_{S_{k+1}}} \circ (\tau f_k(\cdot)+ \id) (x) - x|^2 d\rhoe(x) \label{eq:w2_bounds_1} . 
		\end{gather}
		Next we use the triangle inequality to obtain
		\begin{multline}
		W_2^2(\rhoo,\rhoe)\le \int_{S_k} | P_{S_{k+1}} (\tau f_k(x)+ x) - P_{S_{k+1}}(x)|^2 d\rhoe(x) \\  + | {P_{S_{k+1}}}(x) - x|^2 \dd\rhoe(x).
		\end{multline}
		Projection operators on convex sets satisfy the nonexpansive property and we use this fact for the first term. For the second term, we use the definition of the Hausdorff distance to obtain
		\begin{gather} W_2(\rhoo,\rhoe) \le \Big(\int_{S_k} |\tau f_k(x)|^2 d\rhoe(x)\Big)^{1/2} + \tau \dd_H (S_k,S_{k+1}). 
		\end{gather}
		%		where, $S_k = S(k\tau)$ and $S_{k+1} = S((k+1)\tau)$.
		Using  Assumption \ref{assume:1} for the drift term $f_k(x)$, it holds
		\begin{multline} \label{eq:second_order_moment}
		W_2(\rhoo,\rhoe)  \le \Big(\int_{S_k}  \tau^2 |(L_f(1+|x|))^2| \dd\rhoe(x)\Big)^{1/2}  \\+ \tau d_H (S_k,S_{k+1}).
		\end{multline}
		Next, we establish a bound on the first term on the right-hand side of \eqref{eq:second_order_moment}. 
		We know that $\rhoe = {G^k_{\#}\circ G^{k-1}_{\#}\circ\cdots \circ G^{1}}_{\#} \mu_0 = (G^1\circ G^{2}\circ\cdots\circ G^{k})_{\#} \mu_0$.  Let $G^{k \dots 1} : = G^k \circ G^{k-1} \circ \cdots \circ G^1$, then it follows that
		\begin{gather}
		\int_{S_k} (1 + \vert x \vert )^2 d\rhoe(x) = \int_{S_0} (1+ \vert G^{k\dots 1}(x_0)\vert )^2 \dd{\mu}^{\tau}_0(x_0).
		\end{gather}
		Letting $x_k: = G^{k \dots 1} x_0$, it has been shown in \citep[Section~5]{BroTan20} that
		\begin{gather}
		|x_k| \le \wt C_1 |x_0| + \wt C_2%+ d_H(S(0),S(T)) + 2 C_{\beta})(1+e^{C_{\beta}}C_{\beta}) \coloneqq \tilde{C}_{\max}
		\end{gather}
		for some constants $\wt C_1, \wt C_2 > 0$ depending on the system data. Now using this uniform bound on $|x_k|$ in the first term on the right-hand side of \eqref{eq:second_order_moment}, we get
		\begin{align}\nonumber
		& \Big(\int_{S_k} \tau^2 L_f^2(1+|x|)^2 \dd\rhoe(x)\Big)^{1/2} \\ \nonumber
		&\qquad = \tau L_f	\Big(\int_{S_0} (1+|x_k|)^2 \dd\mu^{\tau}_0(x_0)\Big)^{1/2} \\ 
		&\qquad \le \tau L_f \Big(\int_{S_0}  (1+ \wt C_1 \vert x_0 \vert + \wt C_2)^2 \dd\mu^{\tau}_0(x_0)\Big)^{1/2} \label{eq:sad1} \\
		& \qquad =: \tau L_f C_{\max} \label{eq:sad2}
		\end{align}
		where the equality between \eqref{eq:sad1} and \eqref{eq:sad2} holds because $\mu^{\tau}_0 \in \P(S_0)$. Thus
		\begin{align*}
		W_2(\rhoo,\rhoe)&\le \tau L_f C_{\max} + \tau d_H(S_k, S_{k+1}) \\ &=\tau( L_f C_{\max}+L_s ) 
		\end{align*}
		where we used  Assumption \ref{assume:2} in the last equality.
	\end{proof}
	
	\textit{(1) Proof of convergence of $\rhot$ to $\mu_t$:}
	The convergence is based on computing the bounds on $W_2(\mu_t^{\tau},\mu_s^{\tau})$ for $s,t \in [0,T]$, where we recall that $\rhot$ are the interpolated measures. This is done by using the characterization of absolutely continuous curves \eqref{eq:AC}, i.e.
	\begin{gather}\label{eq:w2_bounds_5}
	W_2{(\rhot,\mu^{\tau}_{s})} \le \int_s^t \vert (\mu^{\tau})'\vert (r)  \dd r \ 
	\end{gather}
	and using H\"older's inequality leads to
	\begin{align}
	\int_s^t|(\mu^{\tau})'| (r) dr &\le (t-s)^{1/2} \Big(\int_s^t \vert(\mu^{\tau})' (r) \vert^2 \dd r\Big)^{1/2} \\ &\le (t-s)^{1/2} \Big(\sum_k \tau \Big( \frac{W_2(\rhoo,\rhoe)}{\tau}\Big)^{2} \Big)^{\frac{1}{2}}.\label{eq:w2_bounds_6}
	\end{align}
	Using the bounds on $W_2(\rhoo,\rhoe)$ from \eqref{eq:w2_bounds_4}, we get
	\begin{align} \label{eq:w2_bounds_7} \nonumber 
	\sum_k \tau\Big(\frac{W_2(\rhoo,\rhoe)}{\tau}\Big)^{2} &\le \sum_k  (L_s+L_f C_{\max})^2 \tau \\ &= (L_s+L_f C_{\max})^2 T.
	\end{align}
	Substituting \eqref{eq:w2_bounds_7} in \eqref{eq:w2_bounds_6}, we obtain
	\begin{gather}\label{eq:w2_bound_f}
	W_2{(\rhot,\mu^{\tau}_{s})}  \le (t-s)^{1/2}(L_s+L_f C_{\max}) T^{1/2}.
	\end{gather}
	Thus, the curves are uniformly $\frac{1}{2}$H\"older continuous. Moreover for each  $t \in [0,T]$,  $\rhot$ lie in the $\cW_2(\Omega)$ space, which is compact. Thus, we can apply the Ascoli-Arzel\`a theorem (for the H\"older continuous functions) i.e., there exists a subsequence $\tau_j$ for which $\mu^{\tau_j}_t \to \mu_t$ uniformly in $W_2$ space and the limit curve $\mu_t$ is absolutely continuous.
	
	Similar to \eqref{eq:w2_bound_f}, one can derive bounds for $\hmut_t$ and conclude that $\hat{\mu}^{\tau_j}_t \to \hmu_t$. Moreover, the limit curves are the same, since 
	\[
	W_2{(\hat{\mu}^\tau_t,\rhot)}  \le (\tau)^{1/2}(L_f C_{\max} + L_s) T^{1/2}.
	\] 
	The last inequality holds because the curve $\hmut_t$ coincides with $\rhot$ at $k\tau$ and they  are constant on the interval $(k\tau, (k+1)\tau]$. Thus, both  curves converge to the same limit curve $\mu_t$.

	\textit{(2) Proof of convergence of $E^{\tau}_t$:}
	In order to study the convergence properties of the velocity vector, we need to investigate the convergence properties of a family of momentum vectors $E^{\tau}_t = v^{\tau}_t \mu^{\tau}_t$ which is a vector measure\footnote{The space of vector valued measures $\cM^n(\Omega)$ is a normed space dual to $\cC(\Omega;\R^n)$. Under this duality the notion of weak star convergence is defined which further implies that bounded sets in $\cM^n(\Omega)$ are weak star compact.} $E^{\tau}_t \in \cM^n(\Omega)$. 
	We define $m^{\tau}\in \cM^n([0,T]\times\Omega)$ as $m^\tau := v^{\tau}_t \mu^{\tau}_t \dd t$.
	\begin{lemma}
		The norm of $m^{\tau}$ satisfies the following bound:
		\begin{align}\label{eq:uniBoundEtau}
		|m^{\tau}| ([0,T]\times \Omega) \le T^{\frac{3}{2}} (L_f C_{\max}+L_s).
		\end{align}
	\end{lemma}
	\begin{proof}
		By definition
		\begin{align*}
		|m^{\tau}| ([0,T]\times \Omega) &= \int_{[0,T]} \dd t \int_{\Omega} | v^{\tau}_t| d\rhot .
		\end{align*}
		Using the Cauchy-Schwarz inequality and then \eqref{eq:vel_wass_rel} we get 
		\begin{align*}
		|m^{\tau}| ([0,T]\times \Omega)&\le T^ {1/2} \int_{[0,T]} ||v^{\tau}_t||_{L^2(\rhot)} dt\\
		&\le T^ {1/2} \sum_k \frac{W_2(\rhoo,\rhoe)}{\tau}.
		\end{align*}
		Using Lemma \ref{lemma:w2_bound_mu}, we further obtain
		\begin{align*}
		|m^{\tau}| ([0,T]\times \Omega)&\le T^ {1/2} \sum_k ( L_f C_{\max} + L_s )\\
		&= T^{\frac{3}{2}} (L_f C_{\max}+L_s )
		\end{align*}
		which is the desired bound.		
	\end{proof}
	So, $m^{\tau}$ is uniformly bounded and thus compact under weak convergence in the space of vector-valued measures on $[0,T]\times \Omega$.  We conclude that up to a subsequence $m^{\tau} \weak m$ and thus $E^{\tau}_t \weak E_t$.
	For $\hat m^\tau := \hat v_t^\tau \hat \mu_t^\tau$, a similar bound holds, i.e. $|\hat{m}^{\tau}|\le T^{\frac{3}{2}} (L_f C_{\max}+L_s)$.	 Using the same arguments one concludes $\hat{m}^{\tau}\weak \hat{m}$ and thus $\hat{E}^{\tau}_t \weak \hat{E}_t$. Moreover, using 
	\citep[Lemma 8.9]{filipoOT} we conclude that $\hat{E}_t = E_t$.
	
	Next, we discuss about the properties of the limit object $E_t$ and show that $E_t$ is absolutely continuous with respect to $\mu_t$, such that $E_t=v_t\mu_t$, for each $t\in [0,T]$.
	
	\textit{(3) Absolute continuity of $E_t$:} At this point, we recall the properties of then Benamou-Brenier functional $\cB(\mu, E)$ defined as follows.
	For $\mu \in \P(\R^n)$ and $E \in \cM^n(\R^n)$, let  
	\begin{multline*}
	\cB(\mu, E) := \! \! \sup_{\substack{a\in \cC(\R;\R),b \in \cC(\R^n;\R^n) \\ a,b \text{ are bounded}\\ a+\frac{1}{2}|b|^2\ge 0 ~\text{pointwise}}} \int_{\R^n} a(x) \, \dd \mu(x) + \int_{\R^n} b(x) \, \dd E(x)
	\end{multline*}
	which has the following properties:
	\begin{itemize}\label{benamou_functional}
		\item $\cB(\cdot, \cdot)$ is convex, lower semicontinuous, and non-negative;
		\item  $\cB(\mu,E) = \frac{1}{2}\int |v|^2 \dd \mu$ only if $E=v \mu$ is absolutely continuous with respect to $\mu$ and $\cB(\mu,E)$ is $\infty$ otherwise.
	\end{itemize}
	Note that $\cB(\mu_t^{\tau}\dd t,E_t^{\tau}\dd t) = \int_{[0,T]} \int_{\Omega}  |v_t^{\tau}|^2 \dd \rhot \dd t$. Now using the uniform bound on $\vert m^{\tau}\vert $ ( thus on $\int_{[0,T]} \int_{\Omega}  |v_t^{\tau}|^2 \dd \rhot \dd t$) in \eqref{eq:uniBoundEtau} and the lower semi-continuity of $\cB(\cdot,\cdot)$, we get
	\[
	\cB(\mu_t \dd t, \dd m) \le  \liminf_{\tau \to 0} \cB(\rhot \dd t, \dd m^{\tau})< \infty
	\]
	where $\dd m$ represents the limit of $\dd m^\tau$, as $\tau \to 0$. We can now invoke the second property mentioned above which implies that $\dd m$ is absolutely continuous with respect to $\dd\mu_t \dd t$. Thus, there exists $v_t$ such that  $\dd m = v_t \mu_t \dd t$. 	Similarly, $E_t$ is absolutely continuous with respect to $\mu_t$ and $E_t = v_t \mu_t$.
	\begin{comment}
	Using the semi-continuity property, 
	\begin{gather}
	\int_a^b ||v_t||^2 d \mu dt \le \liminf_k \int_a^b ||v^{\tau}_k||^2 d \mu_k dt \le  \int |\mu'|^2(t) dt
	\end{gather}
	Similarly, one can obtain the other direction of inequality by exploiting the convexity property of $\B_p$ functional.
	\end{comment}
	
	(4) \textit{Proof that $v_t(x) \in f(t,x) - \cN_{S(t)}(x)$:} Next we show that the velocity $v_t$ belongs to the admissible set of velocities. By construction of the velocities $v_{k+1}^{\tau}$, for every $x\in S_k$, and every $y \in S_{k+1}$, we have
	\[
	\big\langle y-{P_{S_{k+1}}} \circ (\tau f_k(\cdot)+ \id)(x)),\frac{({P_{S_{k+1}}} \circ (\tau f_k(\cdot)+ \id)(x)-x)}{\tau} \big\rangle \ge 0
	\]
	which results in
	\[ 
	\langle y-{P_{S_{k+1}}} \circ (\tau f_k(\cdot)+ \id)(x)), v^{\tau}_{k+1}(x)-f_k(x) \rangle \ge 0,
	\]
	i.e., $v^{\tau}_{k+1}(x)-f_k(x)$ is in the normal cone to set $S_{k+1}$.	
	The above condition should hold in the integral form for any smooth positive function $h(t,x)$ i.e. 
	\begin{multline}\label{eq:vel_proof}
	\int_{[0,T]}\int_{\Omega} \Big\langle h(t_k,x)(y-{P_{S_{k+1}}} \circ (\tau f_k(\cdot)+ \id)x),\\ (v^{\tau}_{k+1}(x)-f_k(x)) \dd \rhoe(x) \Big\rangle \ge 0, ~~\forall y \in S_{k+1} .
	\end{multline}
	Next we can extend  \eqref{eq:vel_proof} to piecewise constant interpolated curves \eqref{eq:const_trnsptmp} and corresponding velocities \eqref{eq:const_vmp}. We make a passage from discrete time $k\tau$ to $ t \in (k\tau, (k+1)\tau]$ by identifying $v^{\tau}_{k+1} \dd\rhoe =  \hat{v}^{\tau}_t \dd\hmut_t = \dd \hat{E}^{\tau}_t$ and similarly, $\dd \mu_k^\tau = \dd \hat \mu_t^\tau$. Next, we use the convergence results established in the first two steps of the proof, so that, in the limit as $\tau \to 0$, we obtain
	\begin{multline*}
	\int_{[0,T]}\int_{\Omega} \Big\langle h(t,x)\big(y-x\big) , \dd E_t(x) \Big\rangle - \\
	\int_{[0,T]}\int_{\Omega} \Big\langle h(t,x)\big(y-x\big),f(t,x) \dd\mu_t(x) \Big\rangle\ge  0, \quad \forall y \in S(t).
	\end{multline*}
	As we have already established that $E_t=v_t \mu_t$ and since $h(t,x)$ is an arbitrary positive function, we get
	\begin{gather*}
	\big\langle y-x, v(t,x)-f(t,x) \big\rangle \ge 0, \quad \forall\, y \in S(t).
	\end{gather*}
	Thus, $v(t,x) - f(t,x) \in -\cN_{S(t)}(x)$.
	
	\textit{Uniqueness of solutions:~}
	The uniqueness of solutions follows from the same argument as presented in Section~\ref{ce_sweep} based on the superposition principle. Furthermore, it also states that the solution has a representation formula as $\mu_t = {X_t}_{\#} \mu_0$.
\end{proof}

\section{Numerical Results}\label{sec:numerics}
\subsection{Moment-SOS hierarchy}
In Section \ref{sec:ce_di}, we derived the continuity equation  \eqref{eq:ce_nsd} associated with dynamical system \eqref{eq:nsd} using the superposition principle. The problem of computing the evolution of a probability measure  through a dynamical system can be interpreted as a feasibility problem where the feasible set is an affine section (modeled by  the continuity equation) of the cone of non-negative measures (supported on the trajectories). For notational convenience, we rewrite the continuity equation in \eqref{eq:ce_nsd} as 
\begin{gather}\label{eq:numeric_le_eqn}
\partial_t\mu +\nabla\cdot (\overline{f}_{\omega}\mu) + \delta_T \cdot \mu_T = \delta_0 \cdot \mu_0
\end{gather}
where $\overline{f}_{\omega}$ is defined in \eqref{eq:mean_field} and the equation should be understood in the weak sense, i.e. when integrated against sufficiently smooth test functions as in \eqref{eq:le_nsd}. Given an initial distribution $\mu_0$  with $\supp(\mu_0) \subset \cX_0 \subset S(0)$, we formulate the problem of evolution of measures in \eqref{eq:nsd} as the following infinite dimensional optimization problem:
\begin{align}\label{eq:sdp_program}\nonumber
&{\rm Find ~ \mu, \mu_T ~such ~that}\\ \nonumber
& \partial_t\mu +\nabla\cdot (\overline{f}_{\omega}\mu) + \delta_T \cdot \mu_T = \delta_0 \cdot \mu_0 \\ \nonumber
&\mu \ge 0, \mu_0 \ge 0, \mu_T \ge 0\\ %\nonumber
&\supp(\mu) \subset [0,T] \times B, \quad  
\supp(\mu_T) \subset S(T),
\end{align}
where $B = \{ (x,v) | x\in \Omega \subset \R^n , v\in F(x) \}$. This is an infinite dimensional linear program. In this section, we present numerical results obtained using the moment-SOS hierarchy to solve this linear program in terms of approximate moments of measures $\mu$ and $\mu_T$. The moments are generated by integration on a dense set of functions $\phi$ in $\cC^1([0,T],\R^n)$ as in \eqref{eq:ce_nsd}. For notational convenience, we use the monomial basis.

\textbf{Numerical Example:}
Consider the nonsmooth system 
\begin{gather}\label{eq:example2d}
\dot{\mathbf{x}}(t) \in (1,0)-\cN_{S}(\mathbf{x}(t))
\end{gather}
where $S =\{\mathbf{x} = (x_1,x_2) \in \R^2: x^2_1 + x^2_2 \le 1\} $.
For this problem, we decompose the measure solution $\mu$ of \eqref{eq:le_nsd} into two measures: $\mu_{S}$ supported in the disc and $\mu_{\partial S}$ supported on the boundary of the disc. In particular, we let
$\cX_S \coloneqq \supp{(\mu_S)}  =\{0\le t\le 1 \} \times \{(\mathbf{x},\mathbf{v}) :
x^2_1+x^2_2 \le 1, v_1 = 1, v_2 = 0\}$, and $
\cX_{\partial S} \coloneqq \supp{(\mu_{\partial S})} = \{0\le t\le 1 \} \times \{ (\mathbf{x},\mathbf{v}) : x^2_1+x^2_2 = 1, x_1 v_2 = x_2 v_1 \}$.
The continuity equation \eqref{eq:le_nsd} can be expressed as
\begin{multline}\label{eq:numeric_ce_eqn}
\int_{\cX_T} \phi(T,\bx) \dd\mu_T(\bx) - \int_{\cX_0} \phi(0,\bx) \dd\mu_0(\bx)\\ 
= \int\limits_{\cX_S} \Big[\partial_t \phi(t,\bx) + 
\nabla_\bx \phi(t,\bx) \cdot \bv \Big] ~\dd\mu_{S}(t,\bx,\bv)
\\
+\int\limits_{\cX_{\partial S}} \Big[\partial_t \phi(t,\bx) 
+ \nabla_\bx \phi(t,\bx) \cdot \bv \Big] ~\dd\mu_{\partial S}(t,\bx,\bv).
\end{multline}
To approximate the moments of the measures satisfying the above equation we use the monomial basis. Let $R[\bx]$ be the ring of multivariate polynomials and $R_k[\bx] \subset R[\bx]$ be the vector space of polynomials of degree not exceeding $k$. Then the monomial basis of $R_k[\bx]$ can be expressed as
$\phi(t,\bx) := t^a \bx^b \coloneqq t^a x_1^{b_1} x_2^{b_2}..x_n^{b_n} $ where $a+ b_1+b_2+.. +b_n \le k$. 
\begin{proposition}\label{prop:lmi}
	Let $m^T_{a,b} := T^a \int_{\cX_T} \bx^b \dd\mu_T(\bx)$ and $m^0_{a,b} := 0^a \int_{\R^n} \bx^b \dd\mu_0(\bx)$ be the moments of $\mu_T$ and $\mu_0$, respectively. Then, using the monomial basis in \eqref{eq:numeric_ce_eqn}, we get
	\begin{multline}\label{eq:lmi}
	m^T_{a,b} - m^0_{a,b} = am^S_{a-1,b} + am^{\partial S}_{a-1,b} \\+  \sum_{i=1}^n \int_{\cX_{\partial S}} b_i t^a \bx^{b-e_i} v_i \dd\mu_{\partial S}(t,\bx,\bv) \\+
	\sum_{i=1}^n \int_{\cX_S} b_i t^a \bx^{b-e_i} v_i \dd\mu_S(t,\bx,\bv),
	\end{multline}
	where $m^{S}_{a-1,b}:=\int_{\cX_{S}}t^{a-1} \bx^b \dd\mu_S(t,\bx,\bv)$, $m^{\partial S}_{a-1,b}:=\int_{\cX_{\partial S}}t^{a-1} \bx^b \dd\mu_{\partial S}(t,\bx,\bv)$, and $e_i = (0,..,1,..,0)$ is the vector with one at the $i$-th entry.
\end{proposition}
%\begin{proof}
%	The proof is a straightforward extension of \cite[Proposition 4.3]{SOUAIBY2023110836}
%\end{proof}
The moment-SOS hierarchy allows us to evaluate approximate moments $m_{a,b}:=m^S_{a,b}+m^{\partial S}_{a,b}$ and $m^T_{a,b}$ related to occupation measures $\mu$ and $\mu_T$. For details about the moment-SOS hierarchy we refer the readers to \cite[Section 4.3]{SOUAIBY2023110836} or \citep{henrion2020moment}.
The initial distribution is $\mu_0 = \delta_{(0,0.5)}$, i.e. the Dirac measure at coordinates $(0,0.5)$ whose moments $m^0_{a,b}$ are readily available. The simulation results are displayed in Figure \ref{fig:msos}, where we plot  the first order moments of terminal measures $\mu_T$, for different terminal times $T$ and computed for a given relaxation order using \texttt{GloptiPoly} and \texttt{SeDuMi}. We can observe the effect of the boundary on the measure even before it hits the boundary. We attribute this effect to the numerical inaccuracy due to finite order truncation. The approximate moments obtained, sometimes called {pseudo-moments}, may not represent the true moments but the accuracy increases as the relaxation order increases.
\begin{figure}[!tbh]
	\centering
	\includegraphics[width=0.45\textwidth]{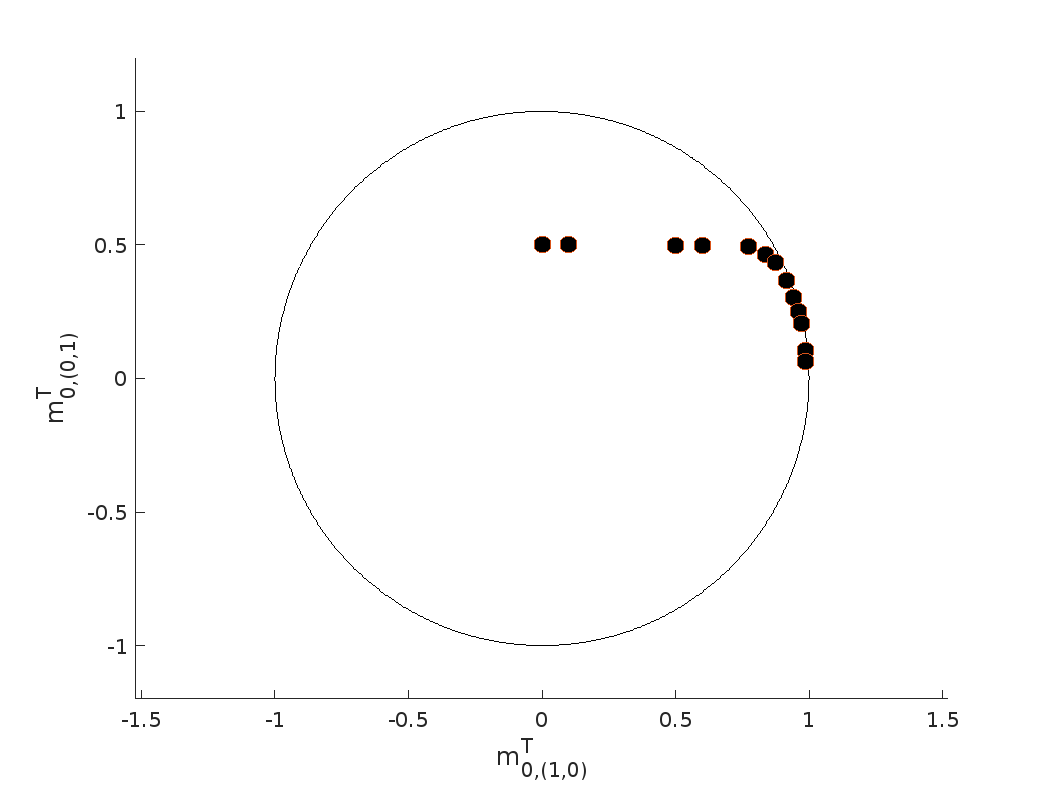}
	\caption{Moment-SOS hierarchy approximations of the first degree moments of the terminal measure $\mu_T$ solving the continuity equation \eqref{eq:numeric_le_eqn} for the nonsmooth system \eqref{eq:example2d}, with different terminal times ranging from $T=0$ (Dirac mass at $(0,0.5)$) to $T=3$ (Dirac mass approximately at $(1,0)$).
		\label{fig:msos}}
\end{figure}

%\vspace{-0.5cm}

\subsection{Bound on Wasserstein Distance}
In Section \ref{sec:fun_reg}, we used the Moreau-Yosida regularization to approximate the solution $(\mu_t)_{t \ge 0}\in\P(\R^n)$ of the continuity equation associated with the nonsmooth dynamical system \eqref{eq:nsd} with measures $( \mul )_{t\ge 0}\in\P(\R^n)$ which are solutions to the continuity equation with regularized vector field \eqref{eq:moreau_yosida}. We showed that the measures  $\mu_t , \mul$, which have representations of the form \eqref{eq:ce} and \eqref{eq:fun_reg_repForm} respectively,
satisfy a bound on their Wasserstein distance of the form $W_1(\mul,\mu_t)  < C_W \sqrt{\lambda}$ for each time $t\in [0,T]$ (the explicit expression of the constant $C_W$ is in \eqref{eq:W1_bound}).
Now let us validate this bound for the following example:
\begin{gather}\label{eq:oneD_ns}
\dot{x}(t) \in -1 - \cN_{\R^+}(x(t)). 
\end{gather}
Its Moreau-Yosida regularization is:
\begin{gather}\label{eq:oneD_reg}
\dot{x}^{\lambda}(t) = -1 -\frac{1}{\lambda}(x^\lambda(t)-\max(x^\lambda(t),0)).
\end{gather}
We consider an initial value problem for \eqref{eq:oneD_ns} and \eqref{eq:oneD_reg} with $x(0) = x^\lambda(0)= 0.5$. The solution to the initial value problem for \eqref{eq:oneD_ns} is 
\begin{gather}\label{eq:sol_ns_oneD}
x(t) = 
\begin{cases} 
0.5 - t & \text{for } 0 < t < \frac{1}{2} \\
0 & \text{for } t \geq \frac{1}{2}
\end{cases}
\end{gather}
and the solution for \eqref{eq:oneD_reg} is 
\begin{gather} \label{eq:sol_reg_oneD}
x^{\lambda}(t) = 
\begin{cases} 
0.5 - t & \text{for } 0 < t < \frac{1}{2} \\
\lambda (\exp(-(t-0.5)/\lambda)-1) & \text{for } t \geq \frac{1}{2}.
\end{cases}
\end{gather}

\begin{proposition}\label{prop:w1Bounds}
	Given $\mul,\mu_t \in \P(\R_+)$ and $\mu_0 = \mu^{\lambda}_0 = \delta_{x=a}$ for some $a>0$, then 
	\begin{gather}
	W_1(\mul,\mu_t) = |\xl(t)-x(t)|. 
	\end{gather}
\end{proposition}
\begin{proof} It is known that in one dimension $W_1(\mu,\nu) = \int_{\R}|F_{\mu}(s)-F_{\nu}(s)| \dd s$ \citep{filipoOT} where $F_{\mu}(\cdot)$ and $F_{\nu}(\cdot)$ are the cumulative distribution functions of measures $\mu$ and $\nu$ respectively. Using this formula, we get
	\begin{gather}\label{eq:w1_oneD}
	W_1(\mul,\mu_t) = \int_{\R}|F_{\mul}(s)-F_{\mu_t}(s)| \dd s.
	\end{gather}
	Given $\mu_0 = \mu^{\lambda}_0 = \delta_{x=a}$, the measures $\mul$ and $\mu_t$ can be computed as  $\mul = \delta_{x^{\lambda}(t)}$ and $\mu_t = \delta_{x(t)}$ and thus $F_{\mul}(s)=\Theta_{\xl(t)}(s)$ and $F_{\mu_t}(s)=\Theta_{x(t)}(s)$. Here, $\Theta:\R \to \R$ is such that $\Theta_{x(t)}(s) = 1 ~{\rm for}~ x(t) \le s$ and $\Theta_{x(t)}(s) = 0 ~{\rm for}~ x(t) \ge s$.
	Substituting $F_{\mul}$ and $F_{\mu_t}$ in \eqref{eq:w1_oneD}, we get
	\begin{gather} \label{eq:w1_oneD2}
	W_1(\mul,\mu_t) =\int_{\R} |\Theta_{x^{\lambda}(t)}(s)-\Theta_{x(t)}(s)| \dd s.
	\end{gather}
	This quantity can be seen as the area under $|\Theta_{x^{\lambda}(t)}(s)-\Theta_{x(t)}(s)|$ at time $t$ and performing the integration \eqref{eq:w1_oneD2} we get the desired result.
\end{proof}

Given  \eqref{eq:sol_ns_oneD}, \eqref{eq:sol_reg_oneD} and using Proposition \ref{prop:w1Bounds} we obtain for $a=0.5$: 
\begin{gather}\label{eq:analyoneD}
W_1(\mul,\mu_t) = 
\begin{cases} 
0 & \text{for } 0 < t < \frac{1}{2} \\
|\lambda (e^{-(t-0.5)/\lambda}-1)| & \text{for } t \geq \frac{1}{2}.
\end{cases}
\end{gather}
Next we obtain the bound in \eqref{eq:W1_bound} for system \eqref{eq:oneD_ns} with $L_f =1 $, $L_s=0$ and $\mu_0 = \delta_{x=0.5}$. Substituting these values we get
\begin{gather}\label{eq:boundoneD}
W_1(\mul,\mu_t) \le \frac{3}{2}e^2 \sqrt{\frac{\lambda}{2}(e^{t}-1)}. 
\end{gather}
In Figure \ref{fig:fun_reg}, we show the plots of the analytical distance \eqref{eq:analyoneD}  and its upper bound \eqref{eq:boundoneD} for different values of $\lambda$ at four different time instants.

\begin{figure}[!tbh]
	\centering
	\includegraphics[scale=.54]{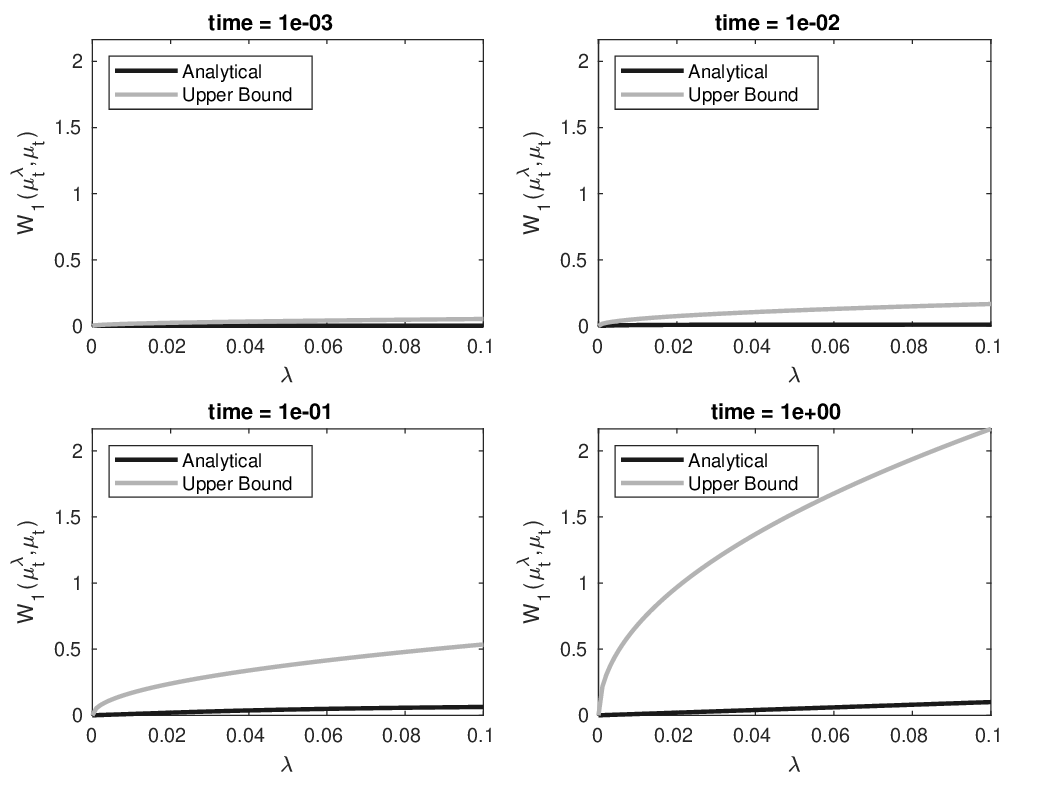}
	\centering \caption{$W_1$ distance \eqref{eq:analyoneD} (in black) and its upper bound \eqref{eq:boundoneD} (in gray) between the image measures  $\mu_t$ resp. $\mul$ through the flow \eqref{eq:ce} resp. the regularized flow \eqref{eq:fun_reg_repForm}, at various time instants. \label{fig:fun_reg}}
\end{figure}

In general, we can also look at the difference between the moments associated with the measures $\mu_t^\lambda$ and $\mu_t$.
\begin{proposition}
Consider the moment sequences $m_k(t)$ resp. $m^{\lambda}_k(t)$ for $\mu_t$ resp. $\mul$. Then, for a fixed $k\in \N$ and $t \in [0,T]$,
	\[	|m_k(t)-m^{\lambda}_k(t)| \le C_k W_1(\mu_t,\mul) \]
	where $C_k := k \max_{z\in \Omega} z^{k-1}$. 
\end{proposition}
\begin{proof}
	By the definition of the moments
	\begin{gather}\label{eq:moment_dist_1}
	|m_k(t)-m^{\lambda}_k(t)| \le \Big|\int_{\Omega} x^k  \dd\mu_t(x) - \int_{\Omega} y^k \dd\mu_t^{\lambda}(y) \Big|.
	\end{gather}
	Let us consider an optimal transport plan $\gamma \in \P(\Omega\times\Omega)$ for compact $\Omega\subset\R^n$ such that ${\pi_x}_{\#}\gamma =  \mu_t$ and ${\pi_y}_{\#}\gamma = \mul$. We use this transport plan  in \eqref{eq:moment_dist_1} to obtain
	\begin{multline}\label{eq:mom_wass}
	\Big|\int_{\Omega} x^k  \dd\mu_t(x) - \int_{\Omega} y^k  \dd\mu_t^{\lambda}(y) \Big| \\ \le \Big|\int_{\Omega} (x^k-y^k) \dd \gamma(x,y)\Big| \le \int_{\Omega} |x^k-y^k| \dd \gamma(x,y).
	\end{multline}
	Next, using the mean value theorem, we have $x^k-y^k = (x-y) k z^{k-1}$ for some $z\in [x,y]$. Using this formula in \eqref{eq:mom_wass}, we obtain
	\begin{multline*}
	|m_k(t)-m^{\lambda}_k(t)|  \le C_k\int_{\Omega} |x-y| \dd \gamma(x,y)  \\= C_k W_1(\mu_t,\mul) \sim O(\sqrt{\lambda})
	\end{multline*}
	where $C_k := k \max_{z\in \Omega} z^{k-1}$.
\end{proof}
So we expect that the distance of the moment sequence corresponding to the measures of the form (37) and (39) is $O(\sqrt\lambda)$. Using the moment-SOS hierarchy for such an approximation scheme as presented in  \cite{SOUAIBY2023110836}, we have a quantitative bound on the numerical error introduced by the functional regularization.

\subsection{Time-Discretized Measure Evolution }
In Section \ref{sec:td_pde}, we used geodesic interpolation and piecewise constant interpolation for the time discretized curves and showed convergence to the solution of the continuity equation \eqref{eq:ce_nsd}. Let us apply the time-stepping algorithm to implement a time-discretized evolution of the measure solutions to the continuity equation associated with the nonsmooth dynamical system \eqref{eq:nsd}.

Let $\tau$ be the time step between two discretized measures and let $\rhoe$ be the measure at time $k$. Then from Section \ref{sec:td_pde}:
\begin{gather}
\mu_0^{\tau} = \mu_0; ~~~~
\mu_{k+1}^{\tau} = {P_{S_{k+1}}}_{\#}  (\tau f_k(\cdot)+ id)_{\#} \mu_k.
\end{gather}
We model the measure as mass distributed on a space discretized grid. Then the time-stepping scheme consists of the following two steps:
\begin{itemize}
	\item Step 1: compute the pushforward measure
	\begin{equation}\label{eq:push_fwd}
	\tilde{\mu}_k \coloneqq 	(\tau f_k(\cdot)+ id)_{\#} \mu_k;
	\end{equation}
	\item Step 2: project each cell lying outside $S_{k+1}$ back onto $S_{k+1}$. This operation can be formulated as
	\begin{gather}\label{eq:proj}
	\mu_{k+1} = {P_{S_{k+1}}}_{\#}\tilde{\mu}_k
	\end{gather}
	and it can be understood as a projection of measures on the set of measures with support $S_{k+1}$.
\end{itemize}

In Figure \ref{fig:time_discrete}, we illustrate the scheme for the following example 
$$
\dot{x}(t) \in -\cN_{S(t)}(x(t))
$$
where $S(t)$ is a $4\times4$ grid square, with each cell assigned with mass of $1/16$ at time $t=0$.
\begin{figure} 
	%	\centering
	\raggedleft
	\includegraphics[scale=.5]{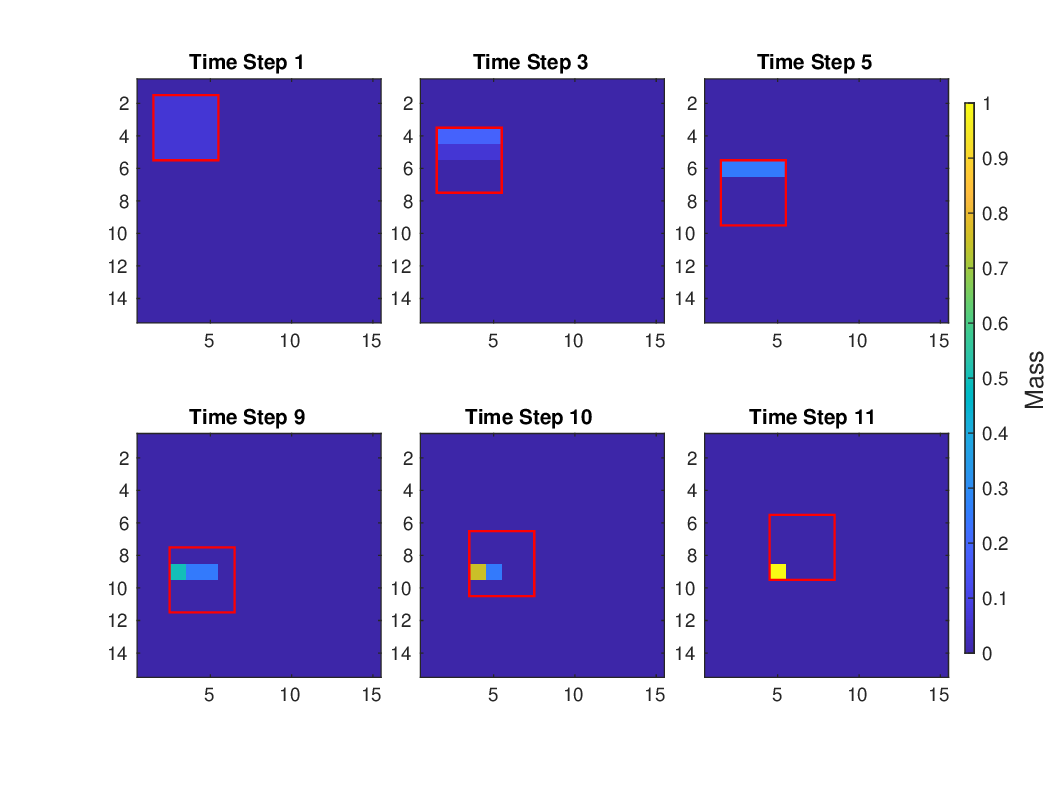}
	\centering \caption{Snapshots of time evolution of a uniform probability distribution on the moving square. The mass is initially distributed on a $4\times 4$ square grid with $1/16$ mass on each cell. For the first $7$ time steps the square moves downward and the mass can be seen to concentrate on the top edge of the square. Afterwards the square moves diagonally and the mass concentrates on the left edge until it completely concentrates at the bottom left corner at time step $12$.  \label{fig:time_discrete}}
\end{figure}

\begin{remark}
If we consider the numerical techniques for each formalism, we observe that each formalism handles the selection of velocity in a different way. This will affect the numerical accuracy and computational efficiency. For the first formalism, we observe that $\mu_S, \mu_{\partial S}$  in (81) are dependent of the velocity vector. Thus, while specifying the support of the measure, the normal cone has to be defined explicitly. Though we can directly write the continuity equation (without using approximation) with this approach, the measure depends on the velocity vector. The increase in number of variables (to be computed) reduces the computational efficiency.

When we approximate the solutions of the continuity equation corresponding to the regularized dynamics (38), the projection of the velocity vector onto the set must be defined. There is no need to select the velocity vector as the vector field is smoothened. However, numerical errors are introduced due to the approximation scheme used.

Similarly, the numerical approximation of the discrete-time problem is based on (95). One needs to define explicitly the projection map onto the set at each time instant. The measures at each time instant do not depend on an extra velocity vector, thus avoiding increasing the number of variables encountered with the first formalism.
\end{remark}

\section{Conclusion}
We addressed the problem of evolution of measures in a nonsmooth dynamical system modeled by evolution variation inequalities using three different formalisms. 
For the time-discretization formalism presented in Section 5, one can also draw some similarities with the literature on constrained sampling for stochastic differential equations, see for example \citep{Bubeck}. In these algorithms, a sequence of random variables is simulated in a recursive manner by projecting one-step of Euler-Maruyama interpolation onto a closed convex set, and under some regularity assumptions, it can be shown that the probability law associated with the limit of these random variables converges to the invariant distribution with respect to total variation. In our case, we restrict our analysis over finite time intervals and compare the time-discretized probability measure with the actual solution without assuming the existence of an invariant distribution.
In future work, we aim to study the optimal control problem for the considered system class using measure relaxation. Based on our preliminary investigation, the results of this article allow us to address the propagation of measures in the presence of control inputs. In particular, it opens up the possibility to study the convergence of discrete time optimal control problem to the continuous time optimal control problem in the space of measures.

\appendix

\section{Bounds used in proof of Proposition~\ref{prop:ce}}\label{appendix:1}
In this appendix, we provide estimates for the bounds $\int|\dot{\gamma}(t)|d\eta(x,\gamma)$. 
To do so, we first get bounds on $|\gamma(t)|$ using the linear growth in Assumption \ref{assume:4},
\begin{align*}\nonumber
|\gamma(t) - \gamma(0)| \le \int | \dot{\gamma}(s)| \dd s \le \int \beta(s)(1+|\gamma(s)|) \dd s \\ \nonumber
\le \int_0^t \beta(s)(1+|\gamma(0)|) \dd s + \int_0^t \beta(s) (|\gamma(s)-\gamma(0)|) \dd s.
\end{align*}
Applying Gronwall inequality by assuming first term is non-decreasing,
\begin{align}
|\gamma(t) - \gamma(0)| \le \Big[\int_0^t \beta(s) (1+|\gamma(0)|)\dd s\Big]~ e^{\Big(\int_0^t \beta(\sigma)\dd \sigma \Big)}.\label{eq:gronwall} 
\end{align}
Let $ \ol \beta_t:= \int_0^t \beta(\sigma) \dd \sigma$, then
\begin{align*}
|\gamma(t) - \gamma(0)| \le \ol\beta_t \exp(\ol\beta_t)(1+|\gamma(0)|) \le \ol\beta_T \exp(\ol\beta_T)(1+|\gamma(0)|).
\end{align*}
We use this last inequality to get the desired bound as follows:
\begin{align*}
& \int_{\R^n\times \Gamma_T^\omega} |\dot{\gamma}(t)| \dd\eta(x,\gamma) \le \int_{\R^n \times \Gamma_T^\omega} \beta(t)(1+|\gamma(t)|) \dd\eta(x,\gamma) \\
& \quad \le \int_{\R^n \times \Gamma_T^\omega} \beta(t)(1+|\gamma(0)|) (\ol\beta_T \exp(\ol\beta_T)+ 1) \dd \eta(x,\gamma) 
\le  K \beta(t),
\end{align*}
where $K := \int_{\R^n \times \Gamma_T^\omega} (1+|\gamma(0)|) (\ol\beta_T \exp(\ol\beta_T)+ 1) \dd \eta(x,\gamma)$.

\section{Differentiating first moment with time dependence}\label{appendix:2}
In this appendix, we derive the Liouville equation \eqref{eq:ce_eqn}. For any $\varphi \in \cC^1(\R,\R^n)$ we differentiate $\int \varphi(t,x) \dd \mu_t(x)$ and get,
\begin{align*}
&\frac{\dd}{\dd t} \int_{\R^n} \varphi(t,x) \dd \mu_t(x) =  \frac{\dd}{\dd t} \int_{\R^n} \varphi(t,x) \dd {\mathbf{e}_t}_{\#}\eta(x,\gamma) \\
& \qquad =\frac{\dd}{\dd t} \int_{\R^n\times \Gamma_T^\omega}  \varphi(t,\gamma(t)) \dd\eta(x,\gamma) \\
& \qquad = \int_{\R^n\times \Gamma_T^\omega} \Big(\partial_t \varphi(t,\gamma(t)) + \nabla_x \varphi(t,x) \dot{\gamma}(t) \Big) \dd\eta(x,\gamma) \\
& \qquad = \int_{\R^n} \Big(\partial_t \varphi(t,x) + \nabla_x \varphi(t,x)\cdot f(t,x) \Big) \dd \mu_t(x).
\end{align*}
Integrating w.r.t. $t$ on both sides gives the desired result.

\section{A property for Moreau-Yosida regularization for the mapping induced by normal cone}\label{appendix:3}
In \citep{SOUAIBY2023110836}, the bound on $|Y_{\lambda}|$ was established as follows,
\begin{gather}\label{apndx:ceq1}
|Y_{\lambda}(\xl(t))| \le \frac{1}{\lambda}\int_0^t e^{-(t-s)/\lambda} (L_f +L_f(|\xl(t)|) + L_s) \dd s.
\end{gather}
It was demonstrated that $\xl$ satisfy uniform bound $|\xl(t)| \le e^{2L_f T}|\xl(0)|+\kappa$ where, $\kappa = (e^{2L_f T}-1) \frac{2 L_f +L_s}{2 L_f}$ and $T$ is the time interval for the existence of trajectory.
Upon substituting this uniform bound on $|\xl|$ in \eqref{apndx:ceq1} we get,
\begin{gather}
|Y_{\lambda}(\xl(t))| \le L_f (1+\kappa +e^{2L_f T}|\xl(0)|) +L_s.
\end{gather}

\bibliography{ref.bib}

\begin{thebibliography}{32}
\expandafter\ifx\csname natexlab\endcsname\relax\def\natexlab#1{#1}\fi
\providecommand{\url}[1]{\texttt{#1}}
\providecommand{\href}[2]{#2}
\providecommand{\path}[1]{#1}
\providecommand{\DOIprefix}{doi:}
\providecommand{\ArXivprefix}{arXiv:}
\providecommand{\URLprefix}{URL: }
\providecommand{\Pubmedprefix}{pmid:}
\providecommand{\doi}[1]{\href{http://dx.doi.org/#1}{\path{#1}}}
\providecommand{\Pubmed}[1]{\href{pmid:#1}{\path{#1}}}
\providecommand{\bibinfo}[2]{#2}
\ifx\xfnm\relax \def\xfnm[#1]{\unskip,\space#1}\fi
%Type = Book
\bibitem[{Adly(2018)}]{Adly2018}
\bibinfo{author}{Adly, S.}, \bibinfo{year}{2018}.
\newblock \bibinfo{title}{A variational approach to nonsmooth dynamics}.
\newblock \bibinfo{publisher}{Springer Briefs in Mathematics. Springer}.
%Type = Incollection
\bibitem[{Ambrosio(2008)}]{Amb08}
\bibinfo{author}{Ambrosio, L.}, \bibinfo{year}{2008}.
\newblock \bibinfo{title}{Transport equation and {Cauchy} problem for
  non-smooth vector fields}, in: \bibinfo{editor}{Dacorogna, B.},
  \bibinfo{editor}{Marcellini, P.} (Eds.), \bibinfo{booktitle}{Calculus of
  Variations and Nonlinear Partial Differential Equations}.
  \bibinfo{publisher}{Springer Berlin Heidelberg}, pp. \bibinfo{pages}{1--41}.
%Type = Book
\bibitem[{Ambrosio et~al.(2005)Ambrosio, Gigli and Savar\'e}]{AmbrGigl05}
\bibinfo{author}{Ambrosio, L.}, \bibinfo{author}{Gigli, N.},
  \bibinfo{author}{Savar\'e, G.}, \bibinfo{year}{2005}.
\newblock \bibinfo{title}{Gradient Flows: {In} Metric Spaces and in the Space
  of Probability Measures}.
\newblock \bibinfo{publisher}{Birkh\"auser}, \bibinfo{address}{Basel}.
%Type = Book
\bibitem[{Aubin and Cellina(1984)}]{aubin_di}
\bibinfo{author}{Aubin, J.P.}, \bibinfo{author}{Cellina, A.},
  \bibinfo{year}{1984}.
\newblock \bibinfo{title}{Differential Inclusions: Set-Valued Maps and
  Viability Theory}.
\newblock \bibinfo{publisher}{Springer-Verlag}, \bibinfo{address}{Berlin,
  Heidelberg}.
%Type = Book
\bibitem[{Aubin and Frankowska(1990)}]{aubin1990}
\bibinfo{author}{Aubin, J.P.}, \bibinfo{author}{Frankowska, H.},
  \bibinfo{year}{1990}.
\newblock \bibinfo{title}{Set-Valued Analysis}.
\newblock \bibinfo{publisher}{Birkh\"{a}user}, \bibinfo{address}{Berlin}.
%Type = Book
\bibitem[{Bensoussan and Lions(1978)}]{BensLion78}
\bibinfo{author}{Bensoussan, A.}, \bibinfo{author}{Lions, J.},
  \bibinfo{year}{1978}.
\newblock \bibinfo{title}{Applications des in\'egalit\'es variationnelles en
  contr\^ole stochastique}.
\newblock \bibinfo{publisher}{Dunod}, \bibinfo{address}{Paris}.
%Type = Article
\bibitem[{Bernardin(2003)}]{Bern03}
\bibinfo{author}{Bernardin, F.}, \bibinfo{year}{2003}.
\newblock \bibinfo{title}{Multivalued stochastic differential equations:
  {C}onvergence of a numerical scheme}.
\newblock \bibinfo{journal}{Set-Valued Analysis} \bibinfo{volume}{11},
  \bibinfo{pages}{393--415}.
%Type = Article
\bibitem[{Bernicot and Venel(2011)}]{bernicot2011}
\bibinfo{author}{Bernicot, F.}, \bibinfo{author}{Venel, J.},
  \bibinfo{year}{2011}.
\newblock \bibinfo{title}{Stochastic perturbation of sweeping process and a
  convergence result for an associated numerical scheme}.
\newblock \bibinfo{journal}{Journal of Differential Equations}
  \bibinfo{volume}{251}, \bibinfo{pages}{1195--1224}.
%Type = Book
\bibitem[{Bogachev(2007)}]{Bogachev_2007}
\bibinfo{author}{Bogachev, V.I.}, \bibinfo{year}{2007}.
\newblock \bibinfo{title}{Measure Theory}.
\newblock \bibinfo{publisher}{Springer Berlin Heidelberg}.
%Type = Article
\bibitem[{Bouchut and James(1998)}]{BoucJame98}
\bibinfo{author}{Bouchut, F.}, \bibinfo{author}{James, F.},
  \bibinfo{year}{1998}.
\newblock \bibinfo{title}{One dimensional transport equation with discontinuous
  coefficients}.
\newblock \bibinfo{journal}{Nonlinear Analysis} \bibinfo{volume}{32},
  \bibinfo{pages}{891--933}.
%Type = Article
\bibitem[{Bouchut et~al.(2005)Bouchut, James and Mancini}]{BoucJameManc05}
\bibinfo{author}{Bouchut, F.}, \bibinfo{author}{James, F.},
  \bibinfo{author}{Mancini, S.}, \bibinfo{year}{2005}.
\newblock \bibinfo{title}{Uniqueness and weak stability for multi-dimensional
  transport equations with one-sided {L}ipschitz coefficients}.
\newblock \bibinfo{journal}{Annali della Scuola Normale Superiore di Pisa -
  Classe di Scienze, S\'erie 5} \bibinfo{volume}{4}, \bibinfo{pages}{1--25}.
%Type = Article
\bibitem[{Brogliato and Tanwani(2020)}]{BroTan20}
\bibinfo{author}{Brogliato, B.}, \bibinfo{author}{Tanwani, A.},
  \bibinfo{year}{2020}.
\newblock \bibinfo{title}{Dynamical systems coupled with monotone set-valued
  operators: {F}ormalisms, applications, well-posedness, and stability}.
\newblock \bibinfo{journal}{SIAM Review} \bibinfo{volume}{62},
  \bibinfo{pages}{3--129}.
%Type = Article
\bibitem[{Bubeck et~al.(2018)Bubeck, Eldan and Lehec}]{Bubeck}
\bibinfo{author}{Bubeck, S.}, \bibinfo{author}{Eldan, R.},
  \bibinfo{author}{Lehec, J.}, \bibinfo{year}{2018}.
\newblock \bibinfo{title}{Sampling from a log-concave distribution with
  projected langevin monte carlo}.
\newblock \bibinfo{journal}{Discrete and Computational Geometry}
  \bibinfo{volume}{59}, \bibinfo{pages}{757--783}.
%Type = Article
\bibitem[{Camlibel and Tanwani(2021)}]{aneel}
\bibinfo{author}{Camlibel, K.}, \bibinfo{author}{Tanwani, A.},
  \bibinfo{year}{2021}.
\newblock \bibinfo{title}{A discretization algorithm for time-varying composite
  gradient flow dynamics}.
\newblock \bibinfo{journal}{IFAC-PapersOnLine} \bibinfo{volume}{54},
  \bibinfo{pages}{558--563}.
\newblock \bibinfo{note}{24th International Symposium on Mathematical Theory of
  Networks and Systems MTNS 2020}.
%Type = Article
\bibitem[{Castaing et~al.(2014)Castaing, Monteiro-Marques and
  de~Fitte}]{castaing2014}
\bibinfo{author}{Castaing, C.}, \bibinfo{author}{Monteiro-Marques, M.D.P.},
  \bibinfo{author}{de~Fitte, P.R.}, \bibinfo{year}{2014}.
\newblock \bibinfo{title}{Some problems in optimal control governed by the
  sweeping process}.
\newblock \bibinfo{journal}{Journal of Nonlinear and Convex Analysis}
  \bibinfo{volume}{15}, \bibinfo{pages}{1043--1070}.
%Type = Article
\bibitem[{Castaing et~al.(2016)Castaing, Monteiro-Marques and
  de~Fitte}]{castaing2016}
\bibinfo{author}{Castaing, C.}, \bibinfo{author}{Monteiro-Marques, M.D.P.},
  \bibinfo{author}{de~Fitte, P.R.}, \bibinfo{year}{2016}.
\newblock \bibinfo{title}{A {S}korokhod problem governed by a closed convex
  moving set}.
\newblock \bibinfo{journal}{Journal of Convex Analysis} \bibinfo{volume}{23},
  \bibinfo{pages}{387--423}.
%Type = Inproceedings
\bibitem[{Cavagnari et~al.(2018)Cavagnari, Marigonda and Piccoli}]{spDI}
\bibinfo{author}{Cavagnari, G.}, \bibinfo{author}{Marigonda, A.},
  \bibinfo{author}{Piccoli, B.}, \bibinfo{year}{2018}.
\newblock \bibinfo{title}{Superposition principle for differential inclusions},
  in: \bibinfo{editor}{Lirkov, I.}, \bibinfo{editor}{Margenov, S.} (Eds.),
  \bibinfo{booktitle}{Large-Scale Scientific Computing},
  \bibinfo{publisher}{Springer International Publishing},
  \bibinfo{address}{Cham}. pp. \bibinfo{pages}{201--209}.
%Type = Inproceedings
\bibitem[{C\'epa(1995)}]{Cepa95}
\bibinfo{author}{C\'epa, E.}, \bibinfo{year}{1995}.
\newblock \bibinfo{title}{\'{E}quations diff\'erentielles stochastiques
  multivoques}, in: \bibinfo{booktitle}{S\'eminaire de Probabalit\'es}, pp.
  \bibinfo{pages}{86--107}.
%Type = Article
\bibitem[{{Di Marino} et~al.(2016){Di Marino}, Maury and
  Santambrogio}]{marino2016}
\bibinfo{author}{{Di Marino}, S.}, \bibinfo{author}{Maury, B.},
  \bibinfo{author}{Santambrogio, F.}, \bibinfo{year}{2016}.
\newblock \bibinfo{title}{Measure sweeping processes}.
\newblock \bibinfo{journal}{Journal of Convex Analysis} \bibinfo{volume}{23},
  \bibinfo{pages}{567--601}.
%Type = Article
\bibitem[{{DiPerna} and Lions(1989)}]{DiPeLion89}
\bibinfo{author}{{DiPerna}, R.J.}, \bibinfo{author}{Lions, P.},
  \bibinfo{year}{1989}.
\newblock \bibinfo{title}{Ordinary differential equations, transport theory and
  {Sobolev} spaces}.
\newblock \bibinfo{journal}{Invent. Math.} \bibinfo{volume}{98},
  \bibinfo{pages}{511--547}.
%Type = Article
\bibitem[{Edmond and Thibault(2006)}]{EdmoThib06}
\bibinfo{author}{Edmond, J.F.}, \bibinfo{author}{Thibault, L.},
  \bibinfo{year}{2006}.
\newblock \bibinfo{title}{{BV} solutions of nonconvex sweeping process
  differential inclusion with perturbation}.
\newblock \bibinfo{journal}{Journal of Differential Equations}
  \bibinfo{volume}{226}, \bibinfo{pages}{135--179}.
%Type = Book
\bibitem[{Henrion et~al.(2020)Henrion, Korda and Lasserre}]{henrion2020moment}
\bibinfo{author}{Henrion, D.}, \bibinfo{author}{Korda, M.},
  \bibinfo{author}{Lasserre, J.}, \bibinfo{year}{2020}.
\newblock \bibinfo{title}{The Moment-{SOS} Hierarchy}.
  volume~\bibinfo{volume}{4} of \textit{\bibinfo{series}{Series on Optimization
  and its Applications}}.
\newblock \bibinfo{publisher}{World Scientific Publishing Europe}.
%Type = Article
\bibitem[{Kantorovich(2006)}]{ksOT}
\bibinfo{author}{Kantorovich, L.}, \bibinfo{year}{2006}.
\newblock \bibinfo{title}{On the translocation of masses}.
\newblock \bibinfo{journal}{Journal of Mathematical Sciences}
  \bibinfo{volume}{133}, \bibinfo{pages}{1381--1382}.
%Type = Article
\bibitem[{McCann(1997)}]{McCann1997ACP}
\bibinfo{author}{McCann, R.J.}, \bibinfo{year}{1997}.
\newblock \bibinfo{title}{A convexity principle for interacting gases}.
\newblock \bibinfo{journal}{Advances in Mathematics} \bibinfo{volume}{128},
  \bibinfo{pages}{153--179}.
%Type = Book
\bibitem[{Monge(1781)}]{monge}
\bibinfo{author}{Monge, G.}, \bibinfo{year}{1781}.
\newblock \bibinfo{title}{M{\'e}moire sur la th{\'e}orie des d{\'e}blais et des
  remblais}.
\newblock \bibinfo{publisher}{De l'Imprimerie Royale},
  \bibinfo{address}{Paris}.
%Type = Book
\bibitem[{Mordukhovich(2006)}]{Mord06}
\bibinfo{author}{Mordukhovich, B.}, \bibinfo{year}{2006}.
\newblock \bibinfo{title}{Variational Analysis and Generalized Differentiation.
  Part~I: Basic Theory}.
\newblock \bibinfo{publisher}{Springer}, \bibinfo{address}{Berlin}.
%Type = Article
\bibitem[{Moreau(1977)}]{moreau1977}
\bibinfo{author}{Moreau, J.J.}, \bibinfo{year}{1977}.
\newblock \bibinfo{title}{Evolution problem associated with a moving convex set
  in a {H}ilbert space}.
\newblock \bibinfo{journal}{Journal of Differential Equations}
  \bibinfo{volume}{26}, \bibinfo{pages}{347--374}.
\newblock \bibinfo{note}{Preliminary version in: Probl\`eme d'\'evolution
  associ\'e \`a un convexe mobile d'un espace {H}ilbertien. Comptes Rendus de
  l'Acad\'emie des Sciences. S\'erie IIb, M\'ecanique, Elsevier, 276: 791-794,
  1973}.
%Type = Book
\bibitem[{Rockafellar and Wets(1998)}]{RockWets98}
\bibinfo{author}{Rockafellar, R.T.}, \bibinfo{author}{Wets, R.J.B.},
  \bibinfo{year}{1998}.
\newblock \bibinfo{title}{Variational Analysis}.
\newblock \bibinfo{publisher}{Springer Verlag}, \bibinfo{address}{Berlin
  Heidelberg}.
%Type = Book
\bibitem[{Santambrogio(2015)}]{filipoOT}
\bibinfo{author}{Santambrogio, F.}, \bibinfo{year}{2015}.
\newblock \bibinfo{title}{Optimal Transport for Applied Mathematicians.
  Calculus of Variations, {PDEs} and Modeling}.
\newblock \bibinfo{publisher}{Birkh\"auser Cham}.
%Type = Article
\bibitem[{Souaiby et~al.(2023)Souaiby, Tanwani and Henrion}]{SOUAIBY2023110836}
\bibinfo{author}{Souaiby, M.}, \bibinfo{author}{Tanwani, A.},
  \bibinfo{author}{Henrion, D.}, \bibinfo{year}{2023}.
\newblock \bibinfo{title}{Ensemble approximations for constrained dynamical
  systems using {Liouville} equation}.
\newblock \bibinfo{journal}{Automatica} \bibinfo{volume}{149},
  \bibinfo{pages}{110836}.
%Type = Article
\bibitem[{Stewart(1990)}]{stewart_LP}
\bibinfo{author}{Stewart, D.E.}, \bibinfo{year}{1990}.
\newblock \bibinfo{title}{A high accuracy method for solving odes with
  discontinuous right-hand side}.
\newblock \bibinfo{journal}{Numerische Mathematik} \bibinfo{volume}{58},
  \bibinfo{pages}{299--328}.
%Type = Book
\bibitem[{Villani(2003)}]{Vil03}
\bibinfo{author}{Villani, C.}, \bibinfo{year}{2003}.
\newblock \bibinfo{title}{Topics in optimal transportation}.
\newblock \bibinfo{publisher}{American Mathematical Society, Providence}.

\end{thebibliography}

\end{document}